\documentclass{svjour3}

\smartqed

\usepackage{amsmath}
\usepackage{amssymb}
\usepackage{graphicx}
\usepackage{hyperref}
\usepackage{cite}

\newcommand{\calo}{\mathcal O}

\newcommand{\bb}{b}
\newcommand{\bc}{c}
\newcommand{\be}{e}
\newcommand{\bq}{q}
\newcommand{\br}{r}
\newcommand{\bt}{t}
\newcommand{\bv}{v}
\newcommand{\bx}{x}
\newcommand{\by}{y}
\newcommand{\zero}{0}
\newcommand{\conj}{\overline}
\newcommand{\C}{\mathbb C}

\newcommand{\ds}{\displaystyle}
\newcommand{\ts}{\textstyle}

\newtheorem{experiment}[theorem]{Experiment}

\title{Computing several eigenvalues of nonlinear eigenvalue problems by selection\thanks{%
Version \today. \\
MH has been supported by an NWO Vidi research grant.
BP has been supported in part by the Slovenian Research Agency (grant P1-0294)
and by an NWO visitor's grant.}}

\author{Michiel~E.~Hochstenbach
\and Bor~Plestenjak}

\institute{M.E.~Hochstenbach \at
Department of Mathematics and Computer Science,
TU Eindhoven, PO Box 513, 5600 MB, The Netherlands,
\url{http://www.win.tue.nl/~hochsten}.
\and
B.~Plestenjak \at
IMFM and Faculty of Mathematics and Physics, University of Ljubljana, Jadranska 19,
SI-1000 Ljubljana, Slovenia, \email{bor.plestenjak@fmf.uni-lj.si}.
}

\authorrunning{Hochstenbach and Plestenjak}

\date{}

\begin{document}
\maketitle

\begin{abstract}
Computing more than one eigenvalue for (large sparse)
one-parameter polynomial and general nonlinear eigenproblems, as well as for
multiparameter linear and nonlinear eigenproblems, is a much harder task
than for standard eigenvalue problems.
We present simple but efficient selection methods based on divided differences to do this.
Selection means that the approximate eigenpair is picked
from candidate pairs that satisfy a certain suitable criterion.
The goal of this procedure is to steer the process away from already detected pairs.
In contrast to locking techniques, it is not necessary to keep converged
eigenvectors in the search space, so that the entire search space may be devoted
to new information.
The selection techniques are applicable to many types of matrix eigenvalue problems;
standard deflation is feasible only for linear one-parameter problems.
The methods are easy to understand and implement.
Although the use of divided differences is well known in the context of nonlinear
eigenproblems, the proposed selection techniques are new for one-parameter problems.
For multiparameter problems, we improve on and generalize our previous work.
We also show how to use divided differences in the framework of homogeneous
coordinates, which may be appropriate for generalized eigenvalue problems
with infinite eigenvalues.

While the approaches are valuable alternatives for one-parameter nonlinear
eigenproblems, they seem the only option for multiparameter problems.

\keywords{
Computing several eigenvalues \and selection \and divided difference
\and deflation \and locking \and homogeneous coordinates \and
quadratic eigenvalue problem \and
polynomial eigenvalue problem \and
nonlinear eigenvalue problem \and
multiparameter eigenvalue problem \and
subspace method \and
Jacobi--Davidson}
\subclass{65F15 \and 65F50 \and 15A18 \and 15A69}
\end{abstract}

\section{Introduction}
\label{sec:intro}
In large sparse matrix eigenvalue problems, a common task is to compute a few eigenvalues
closest to a given target, largest in magnitude, or rightmost in the complex plane.
If we have already (approximately) computed a number of eigenpairs, and would
like to compute a new pair, it is of importance to avoid convergence
to one of the previously computed pairs.

Let $A \in \C^{n \times n}$. For the standard eigenvalue problem
\begin{equation}
A \bx = \lambda \bx,
\label{ep}
\end{equation}
avoidance of previous vectors can be conveniently achieved by computing Schur vectors instead of eigenvectors.
This technique is based on the Schur decomposition $AQ = QR$ for
a matrix $A$, where $Q$ is unitary and $R$ is upper triangular.
If $(\lambda_1, \bq_1)$, \dots, $(\lambda_d, \bq_d)$ are Schur pairs
computed earlier in the process
and $Q_d = [\bq_1 \cdots \bq_d]$, then
$(I-Q_dQ_d^*) \, A \, (I-Q_dQ_d^*)$
has the same Schur pairs as $A$, except for $\lambda_1$, \dots, $\lambda_d$
which are replaced by zero eigenvalues. In a subspace method, the search space may
then be kept orthogonal to $\bq_1$, \dots, $\bq_d$ so that the subspace method
does not notice these zero eigenvalues.
For the generalized eigenvalue problem (GEP) $
A\bx = \lambda B\bx,
$
where $B \in \C^{n \times n}$,
the generalized Schur decomposition for matrix pencils may be exploited
in a similar way.

The Jacobi--Davidson QR (JDQR \cite{SBFV96}) method for \eqref{ep}
and Jacobi--Davidson QZ (JDQZ \cite{SBFV96}) method for GEP
are both examples of methods that are based on the described strategies.
Modified deflation techniques are available for other types of linear eigenvalue
problems such as the (generalized) singular value problem \cite{Hoc01, Hoc09}.

In this paper we discuss a new approach to find several eigenvalues for the
regular nonlinear eigenvalue problem (NEP)
\begin{equation}
\label{nep}
F(\lambda) \, \bx = \zero,
\end{equation}
where $F(\lambda)$ is an $n\times n$ matrix, whose elements are
analytic functions of the complex parameter $\lambda$.
The regularity means that $\det(F(\lambda))$ does not vanish identically.
As a special case, we will consider the polynomial eigenvalue problem (PEP)
\begin{equation}
\label{pep}
P(\lambda) \, \bx=(\lambda^m A_m + \dots + \lambda A_1 + A_0) \, \bx = \zero,
\end{equation}
where all matrices are $n \times n$.
In some applications, the leading matrix $A_m$ may be singular, and eigenvalues
may be infinite. For this reason, it may be beneficial to consider homogeneous coordinates,
which we will do in Section~\ref{sec:homo}.
Moreover, because of the practical importance, as well as for convenience
of presentation, we will focus in particular on the quadratic eigenvalue problem (QEP)
\begin{equation}
\label{qep}
Q(\lambda) \, \bx=(\lambda^2 A + \lambda B + C) \, \bx = \zero.
\end{equation}
Already for the QEP no deflation procedure comparable to those for the standard
and generalized eigenvalue problems is known, which is natural in view of the
existence of $2n$ eigenpairs.
This implies that if we compute eigenvalues with, for instance,
the Jacobi--Davidson method \cite{SBFV96}, we may find the same eigenvalue
again without preventive measures.
It is possible to linearize the QEP or PEP to a GEP, for which a deflation procedure is possible.
However, a clear drawback of this is that linearization increases the dimension
of the problem by a factor $m-1$ (see also the comments in Section~\ref{sec:compare}).
Also, the mathematical properties of linearizations are interesting but not
straightforward; see, e.g., \cite{HMT13}.
For the NEP, linearizations are even more involved.
It is therefore relevant to study techniques that can be directly applied
to the problem at hand.

Several strategies have been mentioned to compute several eigenvalues of nonlinear eigenproblems.
One option to find more than one eigenvalue is {\em locking},
which was studied for the QEP by Meerbergen \cite{Mee01}.
The essence of this approach is to carry out no further computations on sufficiently
converged Schur vectors, and retain them in the search space.
While this method may be effective, a disadvantage
is that the size of the search space grows steadily.
In \cite{Lin86, GLW95a, GLW95b}, 
a technique called {\em nonequivalence deflation}
has been proposed, which replaces the original problem by another.
Kressner \cite{Kre09} has developed a block method, while Effenberger \cite{Eff13}
proposes a deflation strategy for nonlinear eigenproblems.
Several methods have been discussed and compared in \cite{Eff13}
and by G\"{u}ttel and Tisseur \cite{GTi17}.
Some further comparisons can be found in Section~\ref{sec:compare}.

In this paper, we propose an alternative simple and elegant strategy:
computing several eigenvalues by {\em selection}. In each iteration we pick
an approximate eigenpair from (Ritz) pairs that meet a certain selection criterion.
This new strategy is particularly simple to comprehend and implement.
We present several selection methods to compute more than one eigenpair of linear and nonlinear,
and one-parameter and multiparameter eigenvalue problems (MEPs).
We present various variants for the QEP and PEP, and for linear and nonlinear multiparameter
eigenvalue problems.
The selection criteria can be used for all types of eigenvalue problems,
provided that expressions for the divided difference and derivative of the problem
are available; see \eqref{critNEP}.

This work contains contributions in three directions.
Although we have already successfully used some of these selection techniques
in our work on linear multi-parameter eigenvalue problems 
(\cite{HPl02, HKP05}, followed by nonlinear two-parameter eigenproblems
\cite{HMP15, Ple16} and linear three-parameter eigenvalue problems \cite{HMMP19} very recently),
we will present an improvement on these criteria for these problems.
Secondly, to the best of our knowledge, the use of selection techniques
to compute several eigenvalues in the form as described in this paper
is new for one-parameter nonlinear eigenproblems:
the QEP \eqref{qep}, PEP \eqref{pep}, and general NEP \eqref{nep}.
Thirdly, these problems may have infinite eigenvalues when the leading coefficient matrix is singular.
Therefore, methods exploiting homogeneous coordinates may be attractive for these problems
(cf., e.g., \cite{DTi03, HNo07}).
We therefore introduce divided differences in homogeneous coordinates,
which is nontrivial, and new to the best of our knowledge.

Finally, this paper is also meant to serve as an overview paper on selection techniques.
We hope that it will inform about, popularize, and facilitate the use of
these effective and easy-to-implement methods for eigenvalue problems.

There are various subspace expansion methods for eigenvalue problems.
In this paper we will focus on the Jacobi--Davidson method.
However, we want to stress that there are several other options to perform
a subspace expansion, such as nonlinear Arnoldi \cite{Vos04} for NEPs,
and Krylov type methods for MEPs \cite{MPl15}.
The selection techniques operate independently of the expansion of the subspace.

The rest of this paper has been organized as follows.
In Section~\ref{sec:sel} we introduce a new selection criterion for
nonlinear one-parameter eigenvalue problems.
Eigenvalue problems involving matrices which are not full rank
may have infinite eigenvalues. To deal with these in a consistent framework,
homogeneous coordinates are the proper viewpoint; this is studied in
Section~\ref{sec:homo}.
Section~\ref{sec:mep} focuses on the use of selection criteria for
linear and nonlinear MEPs, our original motivation to study these techniques.
We end with some numerical experiments and conclusions in
Sections~\ref{sec:num} and \ref{sec:concl}.

\section{Selection for nonlinear one-parameter eigenvalue problems}
\label{sec:sel}
We first introduce some basic notation and facts.
The pair $(\lambda, \bx)$ is an eigenpair if $F(\lambda)\bx = \zero$
for a nonzero vector $\bx$. If $\by^*F(\lambda)= \zero$ for a nonzero $\by$,
then $\by$ is a left eigenvector for the eigenvalue $\lambda$.
We assume that both $\bx$ and $\by$ have unit norm.
We say that $\lambda_0$ is a simple eigenvalue when
$f(\lambda) := \det(F(\lambda))$ has a simple zero at $\lambda = \lambda_0$.
Neumaier~\cite{Neum85} proves the following proposition about the left
and the right eigenvectors of a simple eigenvalue. The same result
with an alternative proof is presented in \cite{Schr08}.

\begin{proposition}
\label{prop:Fprime}
For the nonlinear eigenvalue problem $F(\lambda) \, \bx= \zero$ the following
are equivalent:
\begin{enumerate}
\item $f(\lambda) = \det(F(\lambda))$ has a simple zero at $\lambda=\lambda_0$;
\item $F(\lambda_0)$ has corank 1, and for right and left eigenvectors $\bx$ and $\by$
corresponding to $\lambda_0$, we have
$\by^*F'(\lambda_0) \, \bx \ne 0$.
\end{enumerate}
\label{prop:neumaier}
\end{proposition}
\smallskip

\noindent
The divided difference for $F$ is defined as
\begin{equation}
\label{divdif}
F[\lambda,\mu]:=\left\{\begin{matrix}
\ds \frac{F(\lambda)-F(\mu)}{\lambda-\mu} & {\rm if}\ \lambda \ne \mu, \\
& \\[-2mm]
F'(\lambda) & {\rm if}\ \lambda=\mu,\end{matrix}\right.
\end{equation}
which is continuous in both variables $\lambda$ and $\mu$.
An alternative way to denote this without distinction of cases is
$F[\lambda,\mu] = \ds \lim_{\mu_1 \to \mu}
\ts \frac{F(\lambda)-F(\mu_1)}{\lambda-\mu_1}$.
We will use similar expressions for convenience and brevity in Section~\ref{sec:mep}.
This divided difference enjoys the following two key properties that can be
used in the selection process.
First, it is easy to see that\begin{equation}
\label{zero}
\by_i^* \, F[\lambda_i,\lambda_j] \, \bx_j = 0,
\end{equation}
when $\bx_j$ is the right eigenvector for $\lambda_j$, and
$\by_i$ is the left eigenvector for a different eigenvalue $\lambda_i \ne \lambda_j$.
Secondly, when $\lambda_i$ is a simple eigenvalue, then it follows from
Proposition~\ref{prop:neumaier} that $\by_i^* \, F[\lambda_i,\lambda_i] \, \bx_i \ne 0$.

Based on these two observations, we develop a new selection criterion for NEPs
using this $F[\cdot,\cdot]$-orthogonality of right and left eigenvectors.
Suppose that we have already computed eigentriplets
$(\lambda_1, \bx_1, \by_1)$, \dots, $(\lambda_d, \bx_d, \by_d)$ for \eqref{nep}.
We assume that all computed eigenvalues $\lambda_1,\dots,\lambda_d$ are simple;
the problem is allowed to have multiple eigenvalues, as long as the
computed ones are simple.
Our selection criteria are not suitable for multiple eigenvalues.

Suppose that $(\theta, \bv)$ is a candidate approximation for the next eigenpair,
where also $\bv$ has unit norm. To steer convergence to a pair different from
the previously detected eigenpairs, and in view of \eqref{zero}, we only consider
approximate eigenpairs for which
$|\by_i^* \, F[\lambda_i, \theta] \, \bv|$ is sufficiently small for $i=1,\dots,d$.
To be precise, in the selection of the candidate approximate eigenpairs we require that
\begin{equation}
\label{critNEP}
\max_{i=1,\dots,d}
\frac{|\by_i^*\, F[\lambda_i,\theta]\, \bv|}{|\by_i^*\, F'(\lambda_i) \, \bx_i|} < \eta,
\end{equation}
where $0 < \eta < 1$ is a fixed constant, which controls the strictness of the selection.
Note that the denominator $|\by_i^*\, F'(\lambda_i) \, \bx_i|$ is a well-known
quantity that arises, e.g., in the eigenvalue condition number (cf.~\cite[Th.~2.20]{GTi17}).

The next proposition explains the behavior of the criterion close to an eigenpair.
Suppose that we have already computed the eigenpair $(\lambda_1, x_1)$,
and now approximate a pair $(\lambda_2, x_2)$.
Let us assume that $(\lambda_2+\varepsilon \phi,
x_2+\varepsilon w)$, for small $\varepsilon$ and vectors of unit norm,
is an approximation for $(\lambda_2, x_2)$.
This is a realistic assumption in this situation;
see \cite{HVo03} for options to extract an approximate eigenvalue from
an approximate eigenvector for the QEP and PEP.

\begin{proposition}\label{proplem2}
Let $y_1\ne 0$ be a left eigenvector for a simple eigenvalue $\lambda_1$
of the nonlinear eigenvalue problem $F(\lambda)x=0$.
Let $x_2\ne 0$ be a right eigenvector for a simple eigenvalue $\lambda_2\ne \lambda_1$ and let
$(\theta,v)$, where $\theta=\lambda_2+\varepsilon \phi$ and
$v=x_2+\varepsilon w$, be a candidate for the next eigenpair. Then
\begin{equation}\label{eq:acr1}
{y_1^*F[\lambda_1,\theta]v\over
y_1^*F'(\lambda_1)x_1}=C\varepsilon+{\cal O}(\varepsilon^2),
\end{equation}
where
\begin{equation}\label{eq:acr3}
C={y_1^*F(\lambda_2)w + \phi \, y_1^*F'(\lambda_2)x_2\over (\lambda_2-\lambda_1)(y_1^*F'(\lambda_1)x_1)}.
\end{equation}
\end{proposition}
\begin{proof}
The proposition follows from the Taylor series expansion
\[
y_1^*(F(\lambda_1)-F(\theta))v=
- ( y_1^*F(\lambda_2)w + \phi \, y_1^* F'(\lambda_2) x_2 )\varepsilon +
{\cal O}(\varepsilon^2),
\]
where we take
into account that $y_1^*F(\lambda_1)=0$ and $F(\lambda_2)x_2=0$.
\end{proof}

Proposition \ref{proplem2} indicates that selection based on divided differences may be difficult for
eigenvalues which are very close.
This is in line with the earlier remark that the selection methods are not suited
for multiple eigenvalues.
We will again see the expression for $C$ in Section~\ref{sec:homo}.

Under the assumptions of Proposition \ref{proplem2},
$\frac{|\by_1^*F[\lambda_1,\theta]\, \bv|}{|\by_1^*\, F'(\lambda_1) \, \bx_1|}$
converges to 0
when $(\theta, \bv) \to (\lambda_2, \bx_2)$ and converges to $1$
when $(\theta, \bv) \to (\lambda_1, \bx_1)$. Although this shows
that the selection criteria \eqref{critNEP} may work for any value $\eta$ in the interval $(0,1)$,
we generally recommend to use $\eta \approx 0.1$, to avoid the already computed eigenvalues
while simultaneously avoiding the selection process to be unnecessarily strict.
We will use this value in the experiments in Section~\ref{sec:num}.

The divided difference for the PEP \eqref{pep}
has the following simple explicit expression.

\begin{proposition}
\label{prop:divdifqep}
For the polynomial eigenvalue problem \eqref{pep} we have
\[
P[\lambda,\theta]
=\sum_{i=0}^{m-1}\lambda^i\theta^{m-1-i}A_m
+\sum_{i=0}^{m-2}\lambda^i\theta^{m-2-i}A_{m-1}
+\cdots
+(\lambda+\theta)A_2+A_1.
\]
In particular, the divided difference for the QEP \eqref{qep} is
\begin{equation}
\label{divdifq}
Q[\lambda,\theta] = (\lambda+\theta)A+B.
\end{equation}
\end{proposition}
\begin{proof}
This follows from an easy calculation.
\end{proof}

The selection criterion for a candidate eigenpair $(\theta, \bv)$ for the QEP
using divided differences then becomes
\begin{equation}
\label{QEPcrit}
\max_{i=1,\dots,d}
\frac{|\by_i^*((\lambda_i+\theta)A+B)\bv|}{|\by_i^*(2\lambda_iA+B) \, \bx_i|} < \eta.
\end{equation}
We note that an important key to the success and simplicity of the selection criteria
is the fact that divided differences can be elegantly generalized for many
types of eigenproblems. More specifically, we will describe a homogeneous variant of
divided differences and a selection criteria in Section~\ref{sec:homo},
and a multiparameter variant in Section~\ref{sec:mep}.

%


We now proceed to discuss some practical matters of the proposed techniques.
A main disadvantage of our selection criterion is that one needs the
left eigenvectors corresponding to converged eigenvalues during the process.
For symmetric and Hermitian eigenvalue problems and right-definite
multiparameter eigenvalue problems (see Section~\ref{sec:mep}),
these left eigenvectors come without extra computations.
In other cases, we generally need extra work to compute them.
For large-scale problems, this may comprise several additional matrix-vector
products to solve $\by$ from $F(\lambda)^*\by = 0$.
Pseudocode for this task is given in Algorithm~1.

\noindent
\vrule height 0pt depth 0.5pt width \textwidth \\
{\bf Algorithm~1:}
An iterative approach for computing a null vector of a singular matrix \\[-3mm]
\vrule height 0pt depth 0.3pt width \textwidth \\
{\bf Input: } (Almost) singular $Z$, (random) nonzero starting vector $\by_0$, tolerance $\varepsilon$. \\
{\bf Output: } An approximate null vector $\by$ of $Z$ with $\|Z\by\| \le \varepsilon$. \\
\begin{tabular}{ll}
1: & Compute $\bb = Z\by_0 \ / \ \|Z\by_0\|$ \\
2: & Solve approximately $Z\bx=\bb$ with an iterative method, \\
& \phantom{M} e.g., (preconditioned) GMRES with tolerance $\varepsilon$ \\
3: & Set $\by=(\bx-\by_0) \ / \ \|\bx-\by_0\|$
\end{tabular} \ \\
\vrule height 0pt depth 0.5pt width \textwidth

We note that in some situations we may have a reasonable approximation to
the left eigenvector $\by$ in the process. Moreover, and more importantly,
any available preconditioner will usually be of great help.
For instance, for eigencomputations, often a preconditioner $M \approx F(\tau)$,
where $\tau \in \C$ is the target of interest, is at our disposal.

For problems that are not truly large-scale, we can solve the system
in step 2 in Algorithm 1 with a direct instead of an iterative method.
In particular, this holds for applications
of the selection techniques for multiparameter eigenvalue problems, where
the problem size is often relatively small: here, $p$ vectors
of length $n$ are sought while the total problem size is of size $n^p$,
where $p$ is the number of parameters.

In either case, the extra work to compute the left eigenvectors will often be
relatively small compared with the overall effort of computing the eigenpairs.
In addition, the left eigenvectors are very useful information to determine the
condition numbers of the eigenvalues, an important quantity to assess
the reliability of the eigenvalue.
For instance, for the QEP an absolute eigenvalue condition number is given by Tisseur \cite{Tis00}:
\[
\kappa(\lambda) = \frac{|\lambda|^2 \, \|A\| + |\lambda| \, \|B\| + \|C\|}
{|\by^*Q'(\lambda) \, \bx|},
\]
where we assume that $x$ and $y$ are of unit norm.
For general PEPs there is a straightforward analogous expression \cite{Tis00}.
We will study the extra costs of the selection in more detail later in this section.



As already mentioned in the introduction, the selection criteria presented in this paper
can be combined with appropriate subspace methods, which expand given subspaces,
and extract potential eigenpairs from the subspaces.
As an example, we now give a pseudocode for a Jacobi--Davidson method to compute several
eigenpairs for the polynomial eigenvalue problem in Algorithm~2
(cf.~also \cite{HSl08}).

\noindent
\vrule height 0pt depth 0.5pt width \textwidth \\
{\bf Algorithm~2:}
Jacobi--Davidson type method for computing several eigenpairs for the polynomial eigenvalue
problem \eqref{pep} using selection. \\[-3mm]
\vrule height 0pt depth 0.3pt width \textwidth \\
{\bf Input: } Matrix polynomial $P(\lambda) = \sum_{i=0}^m \lambda^i A_i$,
desired number of eigenpairs $d$, starting vector $\bv$, tolerance $\varepsilon$,
minimum and maximum dimensions of search space {\sf mindim}, {\sf maxdim},
{\bf threshold $\eta$ for \eqref{critNEP}}. \\
{\bf Output: } $d$ approximate eigenpairs. \\[0.5mm] 
\begin{tabular}{ll}
1: & $\bt=\bv, \ V_0 = [\,], \ k=0$ \\
   & {\bf while} less than $d$ eigenpairs detected:\\ 
2: & \phantom{M} $k=k+1$,\quad {\sf rgs}($V_{k-1}, \bt$) $\to V_k$ \\
3: & \phantom{M} Compute $k$th columns of $W_i^{(k)}=A_iV_k$, \ $i=1,\dots,m$ \\
4: & \phantom{M} Compute $k$th rows and columns of $H_i^{(k)}=V_k^*A_iV_k=V_k^*W_i^{(k)}$ \\
5: & \phantom{M} {\bf Select} the best pair $(\theta, \bc)$ using standard, harmonic, or refined \\
& \phantom{MM} Rayleigh--Ritz {\bf among pairs satisfying criterion \eqref{critNEP}} \\
6: & \phantom{M} $\bv=V_k \bc$ \\
7: & \phantom{M} $\br= P(\theta)\bv = \sum_{i=0}^m \theta^i \, W_i\bc$ \\
8: & \phantom{M} {\bf if} $\|\br\| \le \varepsilon$ and \eqref{critNEP} satisfied:
eigenpair found! \\
& \phantom{MM} {\bf Select} next-best candidate pair $(\theta, \bv)$ {\bf satisfying criterion \eqref{critNEP}} \\
& \phantom{MM} {\bf if} $d$ eigenpairs detected, {\bf stop} \\
9: & \phantom{M} {\bf if} $k = {\sf maxdim}$:\\ 
& \phantom{MM} $k = {\sf mindim}$ \\
& \phantom{MM} {\bf Restart} with ${\sf mindim}$ pairs $(\theta,\bv)$, $\bv=V\bc$
{\bf that satisfy criterion \eqref{critNEP}} \\
& \phantom{MMM} (if too few, supplement by pairs that do not satisfy \eqref{critNEP}) \\
10: & \phantom{M} Solve (approximately) $\bt \perp \bv$ from \
$\displaystyle \Big( I-\frac{P'(\theta)\bv\bv^*}{\bv^*P'(\theta)\bv} \Big) \, P(\theta) \, \bt = -\br$ \\
& \phantom{MM} (possibly with preconditioner) \\
\end{tabular} \ \\
\vrule height 0pt depth 0.5pt width \textwidth

\bigskip\noindent
Let us comment on Algorithm~2.
The key steps for the selection process are indicated in boldface.
In step~3, {\sf rgs} denotes repeated Gram--Schmidt orthogonalization
or any other method to compute an orthonormal basis in a stable way.
There are several strategies available for the extraction of an approximate
eigenpair from the search space ${\cal V}_k$ in step~5.
For the extraction of approximate eigenvectors, standard and
harmonic Rayleigh--Ritz are default options \cite{HSl08}, while \cite{HVo03}
discusses choices for an approximate eigenvalue.
We can use the selection criterion independently of these extraction choices.
The key selection approach described in this paper is used in steps~5 and 9.

In step 5, it might happen that none of the Ritz vectors
satisfies the selection criterion. In such a case we pick the best Ritz pair
given the target regardless of the selection criteria and skip step~8 so
that the method does not find the same eigenpair twice. The idea is that as
the subspace expands, Ritz approximations for
other eigenvalues should appear in the extraction.

We now give some more details about the costs of Algorithm~2 with selection,
compared to the original version without it.
Consider the criterion for the polynomial eigenvalue problem of degree $m$.
We look at \eqref{pep} and \eqref{QEPcrit} to understand the costs.
The selection criterion does not require any extra matrix-vector products (MVs),
except for the $m+1$ MVs $\by^*\!A_j$
per detected left eigenvector $\by$, and any MVs necessary to compute this left eigenvector.
The costs of computing $\by$ depends on the problem.
For problems with certain structure, for instance symmetry, the left vector
may come for free.
For nearly symmetric problems, $\bx$ may be a good approximation to $\by$.
For general problems, several steps with an iterative solver may be necessary to
solve for $\by$. A preconditioner that we already have for the eigenvalue problem
will generally be of great help.
As an alternative for not-too-large problems, an exact solve with $P(\theta)$ is often an efficient
method to find a very good approximation to $\by$.

Next, we study the costs of the divided differences.
Per Ritz vector that we choose to test, we need a product of a ``primitive Ritz vector'' $c$
with the $n \times k$ matrix $V_k$ (cost $2nk$), some vector additions (cost $2nm$)
and one inner product for each of the $d$ already detected eigenpairs (cost $2nd$).
This is summarized in the following table.

\begin{tabular}{lll} \\ \hline
Computation & Approx.~costs & Note \\ \hline \rule{0pt}{2.3ex}%
$\by$ & (Variable) & Only when eigentriplet found \\
$\by^*A_j$ & $\calo(nm)$ & Only when eigentriplet found \\
$\bv_j = V_k\bc_j$, $\by^*P[\lambda_i,\theta_j]\bv_j$ & $\calo(n(k+d+m))$ &
Every step, per pair to test \eqref{critNEP} \\
\hline
\end{tabular}

\medskip\noindent
The dominant costs of lines 2 and 3 in this table will usually be $\approx 2nk$ per Ritz pair
that does {\em not} satisfy the criterion \eqref{critNEP}, so that we have to take the next candidate pair.
We have to compute and store the detected left eigenvectors $\by_i$, but in contrast to
locking, no orthogonalization costs with respect to converged vectors is required
as they do not form part of the search space.
Note that the extra costs of the selection procedure may be considered low compared to the
complexity $\calo(nk^2)$ necessary for orthonormalizing the basis.
Also, note that we get valuable extra information about the left eigenvalues and the
condition numbers of the eigenvalues.

\begin{example}
{\rm
To avoid convergence to the same eigenpairs, one may wonder if the
proposed selection criterion based on divided differences \eqref{divdif} is really needed:
could it be an alternative to simply compare the angles of a current approximate
eigenvector $\bv$ with the already detected eigenvectors $\bx_1, \dots, \bx_d$ instead,
and make sure that $\bv$ is sufficiently different?
This simple example for the standard eigenvalue problem shows that this is generally not a
stable approach; the one based on the divided differences is preferable.

Consider the $2 \times 2$ matrix
\[
A = \left[
\begin{array}{cc}
0 & \ \varepsilon \\
0 & \ \delta
\end{array}
\right],
\]
for small $\delta$ and $\varepsilon$, and suppose that the eigenpair
$(0, \be_1)$ has already been found, where $\be_1$ is the first standard basis vector.
The (nonnnormalized) corresponding left eigenvector is $\by = [\delta, \ \, -\varepsilon]^T$.
Since $\delta \approx 0$, $\lambda = 0$ is close to being a multiple eigenvalue,
and therefore the second eigenvector is numerically ill defined:
all vectors $\bv$ with unit norm and associated Rayleigh quotient $\theta = \bv^*\!A\bv$
have a small residual $\br = A\bv - \theta \bv$.
Comparing the angle of $\bv$ with $\be_1$ only rules out $\bv$ that are close to $\be_1$.
With the divided difference requirement $\by^*\bv \approx 0$ the vector $\bv$
is forced to be close to the true (non-normalized) second eigenvector $[\varepsilon, \ \, \delta]^T$.
}
\end{example}

\section{Selection for polynomial eigenproblems in homogeneous coordinates}
\label{sec:homo}
The context of this section is restricted to polynomial eigenvalue problems
\eqref{pep} rather than the general nonlinear eigenvalue problem.
The PEP \eqref{pep} may have infinite eigenvalues when
the leading matrix $A_m$ is singular. We therefore study these problems in homogeneous
coordinates, and propose a new notion for divided differences in this setting.
We consider the QEP \eqref{qep} for the ease of the presentation,
but the techniques carry over to the general polynomial eigenvalue problem \eqref{pep}.

Already at this place, we would like to stress the fact that a homogeneous
approach may be very valuable in itself, apart from the use for infinite eigenvalues
that may be present.
As we will see in this section, homogeneous coordinates lead to a different
selection criterion (see Experiment~6.1 for a numerical example),
that may be seen as a mediator between the selection
criterion for the standard QEP and the reverse QEP.
Homogeneous techniques are mathematically elegant and pleasing,
and account for the fact that eigenvalue problems may be scaled and transformed
in various ways.

Finite and infinite eigenvalues can be elegantly treated together
in one consistent framework by the use of homogeneous coordinates
\begin{equation}
\label{homo}
Q(\alpha, \beta) = \alpha^2 A + \alpha \beta B + \beta^2 C.
\end{equation}
Here, we have the usual conventions: $\lambda = \alpha / \beta$;
$\lambda = \infty$ corresponds to $(\alpha, \beta) = (1,0)$; and
$(\alpha, \beta)$ and $(\gamma \alpha, \gamma \beta)$
represent the same projective number when $\gamma \in \C$ is nonzero.

In the context of divided differences, we have two pairs
$(\alpha_1, \beta_1)$ and $(\alpha_2, \beta_2)$:
a detected eigenvalue $\lambda = \alpha_1/\beta_1$,
and a new approximation to an eigenvalue $\theta = \alpha_2 / \beta_2$.
It is common to normalize (or scale) $|\alpha_1|^2 + |\beta_1|^2 =
|\alpha_2|^2 + |\beta_2|^2 = 1$, which can be done without loss of generality.
Moreover, we might also assume that the $\beta$'s are real, nonnegative,
and in the interval $[0,1]$; however, this turns out to be sometimes undesirable
in our context of divided differences, for the following reason.
When $(\alpha_2, \beta_2) \to (\alpha_1, \beta_1)$
as projective coordinates, we want the componentwise convergence
$\alpha_1 \to \alpha_2$ and $\beta_1 \to \beta_2$ in what is to follow.
This is not satisfied for, e.g., the pairs $(i,\varepsilon)$ and $(1,0)$,
for $\varepsilon \to 0$. The first pair converges to the second pair,
the infinite eigenvalue $1/0$, but there is no component-wise convergence.

Therefore, instead, we scale the pair $(\alpha_1, \beta_1)$ such that
the coordinate with the maximal absolute value is in $[0,1]$.
For convenience of notation, we assume that this maximum coordinate is
$\beta_1$, but this is not a restriction.
As a result, we know that $\beta_1 \ge \frac{1}{2} \sqrt{2}$.
The other pair is then scaled accordingly to the first pair: $\beta_2 \in [0,1]$,
so that (say) $\beta_2 > 0.7$ when $(\alpha_2, \beta_2)$ is close
enough to $(\alpha_1, \beta_1)$.
This scaling ensures that convergence also implies componentwise convergence.

Let $D_{\alpha}$ denote the derivative operator with respect to $\alpha$.
A homogeneous quantity $DQ(\alpha, \beta)$
for the homogeneous problem \eqref{homo}, that has a similar role as
$Q'(\lambda)$ for problem \eqref{qep}, is (see, e.g., \cite[Thm.~3.3]{DTi03})
\begin{equation}
\label{DQ}
DQ(\alpha, \beta) := \conj \beta \, D_{\alpha} Q(\alpha, \beta) - \conj \alpha \, D_{\beta} Q(\alpha, \beta).
\end{equation}
The following key property of \eqref{DQ} will be exploited in the present section.
It is part of \cite[Thm.~3.3]{DTi03}; cf.~Proposition~\ref{prop:Fprime}.

\begin{proposition}
If eigenvalue $(\alpha, \beta)$ is simple with associated right and left eigenvectors
$\bx$ and $\by$ then $\by^* DQ(\alpha, \beta) \, \bx \ne 0$.
\end{proposition}

We note that a vector of the form $DQ(\alpha, \beta) \, \bx$ has also been
exploited in the context of homogeneous Jacobi--Davidson \cite{HNo07}.

Divided differences make a distinction between eigenvectors belonging to the same,
and belonging to different eigenvalues.
We will use these differences to find out whether approximate eigenpairs
are likely to converge to already detected eigenpairs which need to be avoided,
or to new pairs.
Similar to the nonhomogeneous divided difference $Q[\lambda, \mu]$,
we now would like to derive an expression for a divided difference
$Q[(\alpha_1,\beta_1), (\alpha_2,\beta_2)]$ for the homogeneous problem \eqref{homo},
which may be used in the selection process.
This is done in the following definition, which is new to the best of our knowledge.

\begin{definition}
\label{def:homdd}
Let $(\alpha_1,\beta_1)$ and $(\alpha_2,\beta_2)$ be normalized homogeneous coordinates.
We define divided differences in homogeneous coordinates as
\[
Q[(\alpha_1,\beta_1), (\alpha_2,\beta_2)] :=
\left\{\begin{matrix}
\ds{Q(\alpha_1, \beta_1)-Q(\alpha_2, \beta_2)\over \alpha_1 \beta_2-\alpha_2 \beta_1}
& {\rm if}\ (\alpha_1,\beta_1) \ne (\alpha_2,\beta_2), \\
& \\[-2mm]
DQ(\alpha_1, \beta_1) & {\rm otherwise}.
\end{matrix}\right.
\]
\end{definition}

Note that this expression is both elegant and related to the chordal distance.
We recall that the chordal distance of two homogeneous numbers
$(\alpha_1, \beta_1)$ and $(\alpha_2, \beta_2)$ is (see, e.g., \cite[p.~139]{Ste01})
\begin{equation}
\label{chord}
\chi((\alpha_1, \beta_1), (\alpha_2, \beta_2))
= \frac{|\alpha_1 \, \beta_2 - \alpha_2 \, \beta_1|}
{\sqrt{|\alpha_1|^2 + |\beta_1|^2} \, \sqrt{|\alpha_2|^2 + |\beta_2|^2}}.
\end{equation}
Note that this is the sine of the angle between the two projective
numbers interpreted as vectors; cf.~\cite[p.~75]{DTi03}.
With the standard scaling $|\alpha|^2+|\beta|^2=1$ of projective numbers,
this implies that the absolute value of the denominator
in the first case of Definition~\ref{def:homdd} is equal to the chordal distance
\eqref{chord}.

The next result is a justification of Definition~\ref{def:homdd}:
$Q[(\alpha_1,\beta_1), (\alpha_2,\beta_2)]$ is a continuous function of its two
variables.

\begin{proposition}
\label{prop:homdd}
If $(\alpha_2, \beta_2) \to (\alpha_1, \beta_1)$ then
$Q[(\alpha_1,\beta_1), (\alpha_2,\beta_2)] \to DQ(\alpha_1, \beta_1)$.
\end{proposition}
\begin{proof}
Denote $d\alpha = \alpha_2 - \alpha_1$ and $d\beta = \beta_2 - \beta_1$.
In view of the restriction that the numbers are on the complex
unit circle, the following orthogonality condition holds
(cf., e.g., \cite[Eq.~(3)]{DTi03})
\begin{equation}
\label{orthab}
\conj \alpha_1 \, d\alpha + \conj \beta_1 \, d\beta = 0.
\end{equation}
For the numerator of the divided differences we have
\begin{align*}
Q(\alpha_1, \beta_1) - Q(\alpha_2, \beta_2)
& = Q(\alpha_1, \beta_1) - Q(\alpha_2, \beta_1) + Q(\alpha_2, \beta_1) - Q(\alpha_2, \beta_2) \\
& = (\alpha_1-\alpha_2) \, R_1(\alpha_1, \alpha_2, \beta_1)
+ (\beta_1-\beta_2) \, R_2(\alpha_2, \beta_1, \beta_2),
\end{align*}
where $R_1(\alpha_1, \alpha_2, \beta) = (\alpha_1+\alpha_2)A + \beta B$ and
$R_2(\alpha, \beta_1, \beta_2) = \alpha B + (\beta_1+\beta_2)C$.
Note that
$\ds \lim_{\alpha_2 \to \alpha_1} R_1(\alpha_1, \alpha_2, \beta)
= D_{\alpha}Q(\alpha_1, \beta)$,
$\ds \lim_{\beta_2 \to \beta_1} R_2(\alpha, \beta_1, \beta_2)
= D_{\beta}Q(\alpha, \beta_1)$,
and that such a procedure also extends to general polynomial eigenvalue problems \eqref{pep}.


Assuming $\alpha_1 \ne 0$, we now have, using \eqref{orthab} and the definitions
of $d\alpha$ and $d\beta$,
\begin{align*}
\frac{\alpha_1-\alpha_2}{\alpha_1 \beta_2 - \alpha_2 \beta_1}
& = \frac{-d\alpha}{\alpha_1 \, d\beta-d\alpha \, \beta_1}
= \frac{d\beta \, \conj \beta_1}
{\conj \alpha_1 \, (\alpha_1 \, d\beta+d\beta \, \beta_1 \, \conj \beta_1 \, \conj \alpha_1^{-1})}
= \frac{d\beta \, \conj \beta_1}{d\beta} = \conj \beta_1
\end{align*}
and, similarly,
\begin{align*}
\frac{\beta_1-\beta_2}{\alpha_1 \beta_2 - \alpha_2 \beta_1}
& = \frac{-d\beta}{\alpha_1 \, d\beta-d\alpha \, \beta_1}
= \frac{-d\beta}
{\alpha_1 \, d\beta+d\beta \, \beta_1 \, \conj \beta_1 \, \conj \alpha_1^{-1}}
= \frac{-d\beta \, \conj \alpha_1}{d\beta} = -\conj \alpha_1.
\end{align*}
In the case that $\alpha_1 = 0$, it is easy to check that the first
expression equals $\beta_1^{-1}$ which is equal to $\conj \beta_1$ in this case,
and the second expression equals $d\beta \, d\alpha^{-1} \, \beta_1^{-1} = 0$
and therefore is equal to $-\conj \alpha_1$, since $d\beta$ should vanish
in view of \eqref{orthab}.
\end{proof}


Suppose $(\alpha_1, \beta_1)$ is an eigenvalue with right eigenvector $\bx$
and $(\alpha_2, \beta_2)$ is an eigenvalue with left eigenvector $\by$.
Proposition~\ref{prop:homdd} shows that the following two desirable
properties are satisfied:
\begin{itemize}
\item
$\by^* Q[(\alpha_1,\beta_1), (\alpha_2,\beta_2)] \, \bx \ne 0$
when $\bx$ and $\by$ belong to the same simple eigenvalue
$(\alpha_1,\beta_1) = (\alpha_2,\beta_2)$;
\item $\by^* Q[(\alpha_1,\beta_1), (\alpha_2,\beta_2)] \, \bx = 0$
when $\bx$ and $\by$ correspond to different eigenvalues
$(\alpha_1,\beta_1) \ne (\alpha_2,\beta_2)$.
\end{itemize}
These two properties will help us in the selection process to avoid
convergence towards an already detected eigenvalue:
we would like to only select candidate approximate eigenpairs (in homogeneous form)
$((\theta, \eta), \bv)$ for which the divided differences
$\by_i^* Q[(\alpha_i,\beta_i), (\theta,\eta)] \, \bv$
are small enough for all previously detected eigenvalues $(\alpha_i,\beta_i)$
with corresponding left eigenvectors $\by_i$. For this reason,
and in line with \eqref{critNEP}, during the selection process we require that
\begin{equation}
\label{homselcrit}
\max_{i=1,\dots,d}
\frac{|\by_i^*Q[(\alpha_i, \beta_i), (\theta, \eta)] \, \bv|}
{|\by_i^*DQ(\alpha_i, \beta_i) \, \bx_i|} < \eta,
\end{equation}
where, for instance, $\eta = 0.1$.
Elegantly, the denominator in this criterion also appears in the
denominator of the condition number \cite[Thm.~4.2]{DTi03}.

The above selection criteria may be exploited for polynomial eigenvalue problems
\eqref{pep}, especially if they are expected to have infinite eigenvalues.

We can show that for a PEP the
relation between the standard selection criteria and its
homogeneous counterpart depends only on the magnitude of the eigenvalues. As a result we see that,
if the magnitudes of the new candidate for the eigenvalue and the
computed eigenvalue do not differ much, then \eqref{critNEP} and
\eqref{homselcrit} return very close values and both criteria should give the same decision.

\begin{lemma}\label{lem1} Let $x$ be the eigenvector for a simple eigenvalue
$\lambda$ of the polynomial eigenvalue problem \eqref{pep} of degree $m$.
If $P(\alpha,\beta)$ is the homogeneous variant of \eqref{pep}, i.e.,
$P(\alpha,\beta)=\beta^m P(\alpha/\beta)$ for $\beta\ne 0$, then
\[
P'(\lambda)x=(1+|\lambda|^2)^{(m-2)/2} DP(\alpha,\beta)x,
\]
where
$\alpha= \lambda \, (1+|\lambda|^2)^{-1/2}$ and
$\beta= (1+|\lambda|^2)^{-1/2}$.
\end{lemma}
\begin{proof}From the partial derivatives
\begin{align*}
D_{\alpha}P(\alpha,\beta)&=\beta^{m-1}P'(\alpha/\beta),\\
D_{\beta}P(\alpha,\beta)&=m\,\beta^{m-1}P(\alpha/\beta)-\beta^{m-2}\alpha \, P'(\alpha/\beta),
\end{align*}
it follows from \eqref{DQ} that
\[
DP(\alpha,\beta)x=\beta^{m-2}\,(|\beta|^2+|\alpha|^2)\, P'(\alpha/\beta)\,x
=(1+|\lambda|^2)^{-{(m-2)/2}}\,P'(\lambda)\,x.
\]
\end{proof}

In the next result, we assume for convenience of presentation that
the eigenvalues and approximation are real, as this greatly simplifies
the expressions.

\begin{proposition}\label{lem2}
Let $y_1$ be the left eigenvector of a simple real eigenvalue $\lambda_1$
of the polynomial eigenvalue problem \eqref{pep} of degree $m$ and
let $P(\alpha,\beta)$ be the homogeneous variant of \eqref{pep}.
Let $x_2$ be the right eigenvector for a simple real eigenvalue $\lambda_2\ne \lambda_1$ and let
$(\theta,v) = (\lambda_2+\varepsilon\phi, \, x_2+\varepsilon w)$
be a candidate for the next eigenpair, where $\theta \in \mathbb R$. Then
\begin{equation}\label{eq:cr2}
{y_1^*\,P[(\alpha_1,\beta_1),(\widetilde \alpha_2,\widetilde \beta_2)]\,v\over
y_1^*DP(\alpha_1,\beta_1)x_1}=\left({1+\lambda_1^2\over 1+\lambda_2^2}\right)^{(m-1)/2}C\varepsilon
+ {\cal O}(\varepsilon^2),
\end{equation}
where $C$ is given by \eqref{eq:acr3},
$\alpha_i=\lambda_i\,(1+\lambda_i^2)^{-1/2}$, $\beta_i= (1+\lambda_i^2)^{-1/2}$ for $i=1,2$,
$\widetilde\alpha_2=\theta\,(1+\theta^2)^{-1/2}$, and $\widetilde\beta_2= (1+\theta^2)^{-1/2}$.
\end{proposition}
\begin{proof}
It follows from an expansion that up to ${\cal O}(\varepsilon^2)$-terms
\begin{align*}
\delta\alpha_2 & := \widetilde \alpha_2 - \alpha_2 =
{\phi \over (1+\lambda_2^2)^{3/2}} \,\varepsilon = {\phi \over 1+\lambda_2^2}\,\beta_2\,\varepsilon,\\
\delta\beta_2 & := \widetilde \beta_2 - \beta_2
= -{\lambda_2 \, \phi \over (1+\lambda_2^2)^{3/2}} \,\varepsilon
= -{\phi \over 1+\lambda_2^2} \,\alpha_2\,\varepsilon.
\end{align*}
For the numerator of \eqref{eq:cr2} a multivariate Taylor series expansion gives,
omitting the ${\cal O}(\varepsilon^2)$-terms,
\begin{align*}
y_1^*&P[(\alpha_1,\beta_1),(\widetilde \alpha_2,\widetilde \beta_2)]\,v\\
&= -\frac{y_1^*P(\alpha_2,\beta_2)w \, \varepsilon + y_1^*(\delta\alpha_2 \, D_\alpha P(\alpha_2,\beta_2)
+ \delta\beta_2 \, D_\beta P(\alpha_2,\beta_2))\,x_2}{\alpha_1\widetilde\beta_2 - \widetilde\alpha_2 \beta_1} \\
&=
{(1+\lambda_1^2)^{1/2} \, (1+\lambda_2^2)^{1/2}\over \lambda_2-\lambda_1}
\ y_1^*\big(P(\alpha_2,\beta_2)w +
{\phi\over 1+\lambda_2^2} \, DP(\alpha_2,\beta_2)\,x_2\big)\,\varepsilon\\
&=
{(1+\lambda_1^2)^{1/2}\over
(\lambda_2-\lambda_1)(1+\lambda_2^2)^{(m-1)/2}}
\left(y_1^*P(\lambda_2)w + \phi \, y_1^*P'(\lambda_2)\,x_2\right) \varepsilon,
\end{align*}
where we have applied Lemma \ref{lem1}. If we also use Lemma~\ref{lem1}
for the denominator of \eqref{eq:cr2} and combine the results,
we obtain \eqref{eq:cr2}.
\end{proof}

Finally, we give a result that elegantly connects the homogeneous divided differences
with divided difference of the standard QEP and of the {\em reverse QEP}
defined by $\lambda^2 Q(\lambda^{-1}) = A + \lambda B + \lambda^2 C$.
Let $\lambda$ and $\theta$ be given in homogeneous coordinates by
$(\alpha_1, \beta_1)$ and $(\alpha_2, \beta_2)$, respectively.
By Proposition~\ref{prop:divdifqep}, the divided difference expression for the QEP is
$(\lambda+\theta)A+B$, which in homogeneous coordinates corresponds to
\[
D_1 = (\alpha_1 \beta_2 + \alpha_2 \beta_1) A + \beta_1 \beta_2 B.
\]
A divided difference expression for the reverse QEP is
$B + (\lambda^{-1} + \theta^{-1}) C$, or
\[
D_2 = \alpha_1 \alpha_2 B + (\alpha_1 \beta_2 + \alpha_2 \beta_1) C
\]
in homogeneous coordinates.
The numerator of the homogeneous divided differences as defined in Definition~\ref{def:homdd} is
\[
D = (\alpha_1^2-\alpha_2^2)A + (\alpha_1 \beta_1-\alpha_2 \beta_2)B + (\beta_1^2-\beta_2^2)C.
\]
It may be checked that
\[
D = \frac{\alpha_1^2-\alpha_2^2}{\alpha_1 \beta_2 + \alpha_2 \beta_1} \, D_1
+ \frac{\beta_1^2-\beta_2^2}{\alpha_1 \beta_2 + \alpha_2 \beta_1} \, D_2.
\]
Therefore, the homogeneous approach may be viewed as a mediator
between the divided differences of the QEP and reversed QEP.
(We note that in homogeneous Jacobi--Davidson \cite{HNo07}, the homogeneous vector
used in the subspace expansion has also been shown to be a mediator
between those arising in the standard and reverse QEP.)
In fact, it may be shown that a similar expression also holds for the PEP \eqref{pep}.
This result illustrates a mathematically pleasant property
of homogeneous coordinates.


\section{Multiparameter eigenvalue problems}
\label{sec:mep}
As discussed in the introduction, linear two-parameter eigenvalue problems
have been the origin of our interest in selection criteria \cite{HPl02, HKP05}.
We briefly review various previous results, whereby we also improve
on our previously proposed criteria. We will keep the discussion as concise as possible,
referring to the given references for more information.

Consider the linear two-parameter eigenvalue problem
\begin{align*}
(A_1 - \lambda B_1 -\mu C_1) \, \bx_1 & = \zero, \\
(A_2 - \lambda B_2 -\mu C_2) \, \bx_2 & = \zero,
\end{align*}
where the task is to find one or more eigenvalues $(\lambda, \mu)$ together with
their eigenvectors of the form $\bx_1 \otimes \bx_2$.
We first briefly follow \cite{HKP05}.
Let $\Delta_0 = B_1 \otimes C_2 - C_1 \otimes B_2$, which we assume to be nonsingular.
A left eigenvector $\by_1 \otimes \by_2$ and right eigenvector
$\widetilde \bx_1 \otimes \widetilde \bx_2$
corresponding to different simple eigenvalues $(\lambda_1, \mu_1)$
and $(\lambda_2, \mu_2)$, respectively, are $\Delta_0$-orthogonal:
\[
(\by_1 \otimes \by_2)^* \Delta_0 (\widetilde \bx_1 \otimes \widetilde \bx_2)
= (\by_1^* B_1 \widetilde \bx_1)(\by_2^* C_2 \widetilde \bx_2)
  - (\by_1^* C_1 \widetilde \bx_1)(\by_2^* B_2 \widetilde \bx_2) = 0.
\]
For a selection criterion, we would like an approximate eigenvector
$\bv_1 \otimes \bv_2$ to be sufficiently $\Delta_0$-orthogonal to already
detected left eigenvectors $\by_1^{(i)} \otimes \by_2^{(i)}$, $i = 1, \dots, d$.
In our previous criterion, as was proposed in \cite{HKP05}, we required
potential $\bv_1 \otimes \bv_2$ to satisfy
\begin{equation}
\label{mepstrict}
\max_{i=1,\dots,d} \ (\by_1^{(i)} \otimes \by_2^{(i)})^* \Delta_0 (\bv_1 \otimes \bv_2)
< \tfrac{1}{2} \cdot \min_{i=1,\dots,d} \
(\by_1^{(i)} \otimes \by_2^{(i)})^* \Delta_0 (\bx_1^{(i)} \otimes \bx_2^{(i)}).
\end{equation}
While criterion \eqref{mepstrict} has turned out to perform satisfactorily in the
numerical tests in \cite{HPl02, HKP05},
it may be unnecessarily strict:
if one eigenvalue has been detected with right and left eigenvector
$\bx_1 \otimes \bx_2$ and $\by_1 \otimes \by_2$
for which the right-hand side of \eqref{mepstrict}
is small, the selection procedure may reject many or all candidate Ritz pairs.

Therefore, instead of \eqref{mepstrict}, we propose the new modified criterion
(cf.~\eqref{critNEP})
\begin{equation}
\label{mepnew}
\max_{i=1,\dots,d} \
\frac{(\by_1^{(i)} \otimes \by_2^{(i)})^* \Delta_0 (\bv_1 \otimes \bv_2)}
{(\by_1^{(i)} \otimes \by_2^{(i)})^* \Delta_0 (\bx_1^{(i)} \otimes \bx_2^{(i)})}
< \eta,
\end{equation}
with, e.g., $\eta = 0.1$.
This criterion has been successfully used very recently in \cite{HMMP19}.

In \cite{HPl02} the special but important right-definite case has been treated,
where all matrices $A_i$, $B_i$, and $C_i$ are Hermitian, and
$\Delta_0$ is positive definite. In this situation, the right and left eigenvectors coincide,
and therefore eigenvectors $\bx_1 \otimes \bx_2$ and
$\widetilde \bx_1 \otimes \widetilde \bx_2$
corresponding to different eigenvalues are $\Delta_0$-orthogonal:
$(\bx_1 \otimes \bx_2)^* \Delta_0 (\widetilde \bx_1 \otimes \widetilde \bx_2) = 0$.
We note that \cite{HPl02} has been the first paper where a selection criterion
to compute several eigenvalues has been proposed and used,
in the context of linear two-parameter eigenproblems.

Besides being simple and easy to implement, selection criteria can be elegantly
extended to other types of (multiparameter) eigenvalue problems.
The $\Delta_0$-orthogonality can be nicely extended to
nonlinear two-parameter eigenvalue problems as follows.

For the polynomial and general nonlinear two-parameter eigenvalue problem
\begin{align*}
T_1(\lambda, \mu) \, \bx_1 & = \zero, \\
T_2(\lambda, \mu) \, \bx_2 & = \zero,
\end{align*}
we have introduced in \cite{HMP15} a generalized divided difference
\[
T[(\lambda_1, \mu_1), (\lambda_2, \mu_2)]
 = \left| \begin{array}{ccc}
\ds \lim_{\lambda\to \lambda_2} \ts
\frac{T_1(\lambda, \mu_1) - T_1(\lambda_1, \mu_1)}{\lambda-\lambda_1} &
\ \ds \lim_{\mu\to \mu_2} \ts
\frac{T_1(\lambda_2, \mu) - T_1(\lambda_2, \mu_1)}{\mu-\mu_1} \\
\ \ds \lim_{\lambda\to \lambda_2} \ts
\frac{T_2(\lambda, \mu_1) - T_2(\lambda_1, \mu_1)}{\lambda-\lambda_1} &
\ \ds \lim_{\mu\to \mu_2} \ts
\frac{T_2(\lambda_2, \mu) - T_2(\lambda_2, \mu_1)}{\mu-\mu_1}
\end{array} \right|_\otimes,
\]
where \mbox{\scriptsize $\left|\!\!
\begin{array}{cc}
A& B\\ C & D
\end{array}
\!\!\right|_\otimes$} stands for the operator determinant $A\otimes D - B \otimes C$;
see also \cite{Ple16}.
In these papers, it has been shown that this divided difference
has the desired property that the quantity
\[
(\by_1^{(i)} \otimes \by_2^{(i)})^* \, T[(\lambda_i, \mu_i), (\theta, \eta)]
\, (\bv_1 \otimes \bv_2)
\]
is nonzero when $(\theta, \eta)$ converges to $(\lambda_i, \mu_i)$,
while it is 0 when the pair converges to another eigenvalue.

We now illustrate the adaptivity and flexibility of the selection criterion
by the following generalization for the differentiable nonlinear three-parameter eigenvalue problem
\begin{align}
T_1(\lambda, \mu, \nu) \, \bx_1 & = \zero, \nonumber \\
T_2(\lambda, \mu, \nu) \, \bx_2 & = \zero, \label{3EP} \\
T_3(\lambda, \mu, \nu) \, \bx_3 & = \zero, \nonumber
\end{align}
While the special case of a linear case of this problem
has been treated recently in \cite[Lem.~4.2]{HMMP19},
we now define a divided difference for the nonlinear case \eqref{3EP}.

\begin{definition}
We define the divided difference $T[(\lambda_1, \mu_1, \nu_1), (\lambda_2, \mu_2, \nu_2)]$
for problem \eqref{3EP} by
\[
{\scriptsize
\left| \begin{array}{ccc}
\ds \lim_{\lambda\to \lambda_2} \ts \frac{T_1(\lambda,\mu_1,\nu_1) - T_1(\lambda_1,\mu_1,\nu_1)}{\lambda-\lambda_1} &
\ \ds \lim_{\mu\to \mu_2} \ts \frac{T_1(\lambda_2,\mu,\nu_1) - T_1(\lambda_2,\mu_1,\nu_1)}{\mu-\mu_1} &
\ \ds \lim_{\nu\to \nu_2} \ts \frac{T_1(\lambda_2,\mu_2,\nu) - T_1(\lambda_2,\mu_2,\nu_1)}{\nu-\nu_1} \\
\ds \lim_{\lambda\to \lambda_2} \ts \frac{T_2(\lambda,\mu_1,\nu_1) - T_2(\lambda_1,\mu_1,\nu_1)}{\lambda-\lambda_1} &
\ \ds \lim_{\mu\to \mu_2} \ts \frac{T_2(\lambda_2,\mu,\nu_1) - T_2(\lambda_2,\mu_1,\nu_1)}{\mu-\mu_1} &
\ \ds \lim_{\nu\to \nu_2} \ts \frac{T_2(\lambda_2,\mu_2,\nu) - T_2(\lambda_2,\mu_2,\nu_1)}{\nu-\nu_1} \\
\ds \lim_{\lambda\to \lambda_2} \ts \frac{T_3(\lambda,\mu_1,\nu_1) - T_3(\lambda_1,\mu_1,\nu_1)}{\lambda-\lambda_1} &
\ \ds \lim_{\mu\to \mu_2} \ts \frac{T_3(\lambda_2,\mu,\nu_1) - T_3(\lambda_2,\mu_1,\nu_1)}{\mu-\mu_1} &
\ \ds \lim_{\nu\to \nu_2} \ts \frac{T_3(\lambda_2,\mu_2,\nu) - T_3(\lambda_2,\mu_2,\nu_1)}{\nu-\nu_1}
\end{array} \right|_\otimes.
}
\]
\end{definition}

The following results justify this definition.

\begin{proposition}
The quantity
\[
(\by_1 \otimes \by_2 \otimes \by_3)^* \,
T[(\lambda_1, \mu_1, \nu_1), (\lambda_2, \mu_2, \nu_2)] \,
(\bx_1 \otimes \bx_2 \otimes \bx_3)
\]
is nonzero when $(\lambda_2, \mu_2, \nu_2) = (\lambda_1, \mu_1, \nu_1)$ is a simple
eigenvalue with right and left eigenvector $\bx_1 \otimes \bx_2 \otimes \bx_3$
and $\by_1 \otimes \by_2 \otimes \by_3$, respectively;
it equals 0 when $\bx_1 \otimes \bx_2 \otimes \bx_3$ and
$\by_1 \otimes \by_2 \otimes \by_3$ belong to different eigenvalues
$(\lambda_2, \mu_2, \nu_2) \ne (\lambda_1, \mu_1, \nu_1)$.
\end{proposition}
\begin{proof}
A rather straightforward generalization of \cite[Prop.~3.2]{MPl09} shows that
\begin{align*}
& (\by_1 \otimes \by_2 \otimes \by_3)^* \,
T[(\lambda_1, \mu_1, \nu_1), (\lambda_1, \mu_1, \nu_1)] \,
(\bx_1 \otimes \bx_2 \otimes \bx_3) \\[1mm]
& \phantom{M} =
{\footnotesize
\left|
\begin{array}{ccc}
\by_1^* \frac{\partial T_1}{\partial \lambda} \bx_1 &
\quad \by_1^* \frac{\partial T_1}{\partial \mu} \bx_1 &
\quad \by_1^* \frac{\partial T_1}{\partial \nu} \bx_1 \\[1mm]
\by_2^* \frac{\partial T_2}{\partial \lambda} \bx_2 &
\quad \by_2^* \frac{\partial T_2}{\partial \mu} \bx_2 &
\quad \by_2^* \frac{\partial T_2}{\partial \nu} \bx_2 \\[1mm]
\by_3^* \frac{\partial T_3}{\partial \lambda} \bx_3 &
\quad \by_3^* \frac{\partial T_3}{\partial \mu} \bx_3 &
\quad \by_3^* \frac{\partial T_3}{\partial \nu} \bx_3 \\
\end{array}
\right| \ne 0.
}
\end{align*}
When $\lambda_2 \ne \lambda_1$, $\mu_2 \ne \mu_1$, and $\nu_2 \ne \nu_1$,
\begin{align*}
& (\by_1 \otimes \by_2 \otimes \by_3)^* \,
T[(\lambda_1, \mu_1, \nu_1), (\lambda_2, \mu_2, \nu_2)] \,
(\bx_1 \otimes \bx_2 \otimes \bx_3) = \\[1mm]
& \phantom{M}(\lambda_2-\lambda_1)^{-1}(\mu_2-\mu_1)^{-1}(\nu_2-\nu_1)^{-1} \cdot \\[1mm]
& \phantom{MM} {\scriptsize
\left| \begin{array}{ccc}
\by_1^*T_1(\lambda_2,\mu_1,\nu_1)\bx_1 &
\quad \by_1^*(T_1(\lambda_2,\mu_2,\nu_1) - T_1(\lambda_2,\mu_1,\nu_1))\bx_1 &
\quad -\by_1^*T_1(\lambda_2,\mu_2,\nu_1)\bx_1 \\[1mm]
\by_2^*T_2(\lambda_2,\mu_1,\nu_1)\bx_2 &
\quad \by_2^*(T_2(\lambda_2,\mu_2,\nu_1) - T_2(\lambda_2,\mu_1,\nu_1))\bx_2 &
\quad -\by_2^*T_2(\lambda_2,\mu_2,\nu_1)\bx_2 \\[1mm]
\by_3^*T_3(\lambda_2,\mu_1,\nu_1)\bx_3 &
\quad \by_3^*(T_3(\lambda_2,\mu_2,\nu_1) - T_3(\lambda_2,\mu_1,\nu_1))\bx_3 &
\quad -\by_3^*T_3(\lambda_2,\mu_2,\nu_1)\bx_3 \\[1mm]
\end{array} \right| = 0,
}
\end{align*}
since the sum of the columns is the zero vector.
Finally, when some, but not all, of the coordinates of
$(\lambda_1,\mu_1,\nu_1)$ are equal to $(\lambda_2,\mu_2,\nu_2)$, the determinant vanishes as well.
Indeed, it can be checked that:
\begin{itemize}
\item when $(\lambda_1, \mu_1, \nu_1)$ and $(\lambda_2, \mu_2, \nu_2)$ agree in one of three
coordinates then the columns where the coordinates do not agree differ by a factor $-1$;
\item when $(\lambda_1, \mu_1, \nu_1)$ and $(\lambda_2, \mu_2, \nu_2)$ agree in two of three
coordinates then the column where the coordinates do not agree is zero.
\end{itemize}
\end{proof}

This result implies that selection criteria in the line of \eqref{mepnew} can
be exploited.

For multiparameter eigenvalue problems, locking becomes less and less attractive
as the number of parameters increases. For instance, when we are prepared to solve projected
eigenvalue problems of dimension approximately 100 at the subspace extraction step,
the search spaces are limited to dimension 10 for two parameters, and even to dimension
5 for three parameters.
As locking keeps the converged vectors in the search space, this technique is
generally not an option for MEPs.


Therefore, selection effectively creates more space in the subspaces to contain new information.
However, even when using selection criteria to compute several eigenvalues,
already detected vectors may sometimes turn up in the search space in practice.
Therefore, we will be limited by the size of the search space at some point,
and we cannot expect to compute arbitrarily many eigenvalues.

\section{Comparison with other approaches}
\label{sec:compare}
A good comparison of various approaches has already been carried out in \cite{Eff13}.
Here we briefly discuss differences of selection criteria compared to other methods for
computing several eigenvalues of one-parameter eigenvalue problems.

Besides our selection criteria, there are several alternatives for the computation
of several eigenvalues for nonlinear eigenvalue problems and linear and nonlinear
multiparameter eigenvalue problems.
Nonequivalence deflation \cite{Lin86, GLW95a, GLW95b}
has an elegant mathematical foundation, but changes the
original problem, and might suffer from instabilities.
Block methods may be used to compute several eigenvalues simultaneously \cite{Kre09},
but also has some drawbacks as indicated in \cite{Eff13}.
Locking, which keeps the eigenvectors in the search space \cite{Mee01} \cite[Ch.~6]{Eff13},
leaves less space for new vectors; to find new eigenvectors, the search spaces have to grow.
Especially for multiparameter eigenvalue problems, where the dimension of
the projected problems grows as $n^p$, with $p$ the number of parameters,
locking is not a realistic alternative.

We will now discuss differences with the method by Effenberger \cite{Eff13},
which we consider state-of-the-art and of particular importance, in more detail.
This method, as our approach, also computes the eigenvalues successively while preventing
convergence to the same eigenpairs.
However, this method and the one proposed here are still of very different nature.
First, the method in \cite{Eff13} is far from trivial to implement.
During the computations, it modifies the original problem by adding rows and columns
so that the problem size steadily increases.
It is unable to deal with infinite eigenvalues,
as it does not use homogeneous coordinates.
Moreover, it is an open question if the approach can be generalized to
multiparameter eigenvalue problems.
As a big advantage, Effenberger's method has been designed with the aim of
also computing multiple and clustered eigenvalues in a stable way.
Our proposed approach, on the other hand, is (much) simpler, both conceptually
and with respect to implementation (just a few lines of codes on
top of an existing code).
Our method is designed to handle infinite eigenvalues by homogeneous coordinates,
and the problem remains unmodified during the iterations.
Also, the techniques are elegantly generalizable to various types of eigenproblems.
On the other hand, as stated before,
the method is not suitable to compute multiple eigenvalues.

We note that standard deflation methods (see Section~\ref{sec:intro})
for the generalized eigenvalue problems
can be used for polynomial one-parameter and multiparameter eigenvalue problems
when one is prepared to linearize the problem into a (much) larger problem.
For instance, in \cite{MPl15}, a Krylov--Schur type method has been proposed
for the linear two-parameter eigenvalue problem, which works on the operators
$\Delta_0^{-1} \Delta_1$ or $\Delta_0^{-1} \Delta_2$.
A main disadvantage of this approach is that it works on vectors of length
$n^2$, instead of $n$ for a direct approach.
The action with the $\Delta_0^{-1} \Delta_i$ operators can be done in
$\calo(n^3)$ effort instead of the expected $\calo(n^6)$ by solving a Sylvester equation.
Therefore, when $n$ is small enough, this method may still be worthwhile
for two-parameter problems. For three-parameter problems, the situation looks
far less favorable \cite{HMMP19}.

\section{Numerical examples}
\label{sec:num}
We present some numerical examples obtained with Matlab.
Several successful experiments with several types of multiparameter eigenvalue problems
have been carried out and described in \cite{HPl02, HKP05, HMP15, Ple16, HMMP19}.
Therefore, we concentrate ourselves mostly on the new use for polynomial eigenvalue problems.

\begin{experiment}
\label{exp1}
{\rm We consider the QEP {\sf utrecht1331} with target $\tau = -70 - 2000i$
as in \cite{HSl08}, and an exact LU preconditioner based on this target.
This is a quite challenging interior eigenvalue problem, due to the difficult spectrum, the
interior location of the target, and the different scales of the real and imaginary parts;
see Fig.~\ref{fig:1}(a). Approximate eigenpairs $(\theta, \bv)$ are computed to
relative tolerance $10^{-6}$, meaning
\[
\|\br\| := \|Q(\theta)\bv\| \le 10^{-6} \cdot (\,|\theta|^2 \, \|A\|_1 + |\theta| \, \|B\|_1 + \|C\|_1\,).
\]
At first, we take selection threshold $\eta = 0.1$ in \eqref{QEPcrit}.
We use the Jacobi--Davidson method with harmonic extraction \cite{HSl08} and
10 steps of {\sf bicgstab} to solve the correction equations.
The left eigenvectors are solved by an exact solve with $Q(\theta)^*$ when $\theta$
has sufficiently converged.
For the value extraction we use the one-dimensional Galerkin {\sf gal1} approach from \cite{HVo03}.
With minimum and maximum subspace sizes of 20 and 40, we find 12 eigentriplets
in 200 iterations; the convergence history is displayed in Fig.~\ref{fig:1}(b).
The eigenvalues are detected after 10, 12, 14, 16, 78, 96, 105, 118, 133, 147, 165, and 178 iterations.
Elegantly, when we sort the eigenvalues with respect to distance to the target,
these are eigenvalues number 1 through 12, in this order!
The longer ``hiccup'' after several eigenpairs (here the 5th) may occur in many problems,
and is likely due to the fact that new information needs to be inserted in the search space.
We note that it seems important that the search spaces are allowed to be sufficiently large;
otherwise, at some point, the convergence may stop altogether.
For instance, using minimal subspace size 15 and maximal subspace size 25, only 4 eigenvalues are detected
in 200 iterations, with indices 5, 7, 6, and 10.
Favorably, the process seems to be not very sensitive with respect to the precise threshold value of $\eta$:
the choices of $\eta = 0.01$, $0.2$, and $0.5$ result in 9, 13, and 13 found eigenpairs, respectively.

\begin{figure}[!htbp]
\centering
  \includegraphics[scale=0.43]{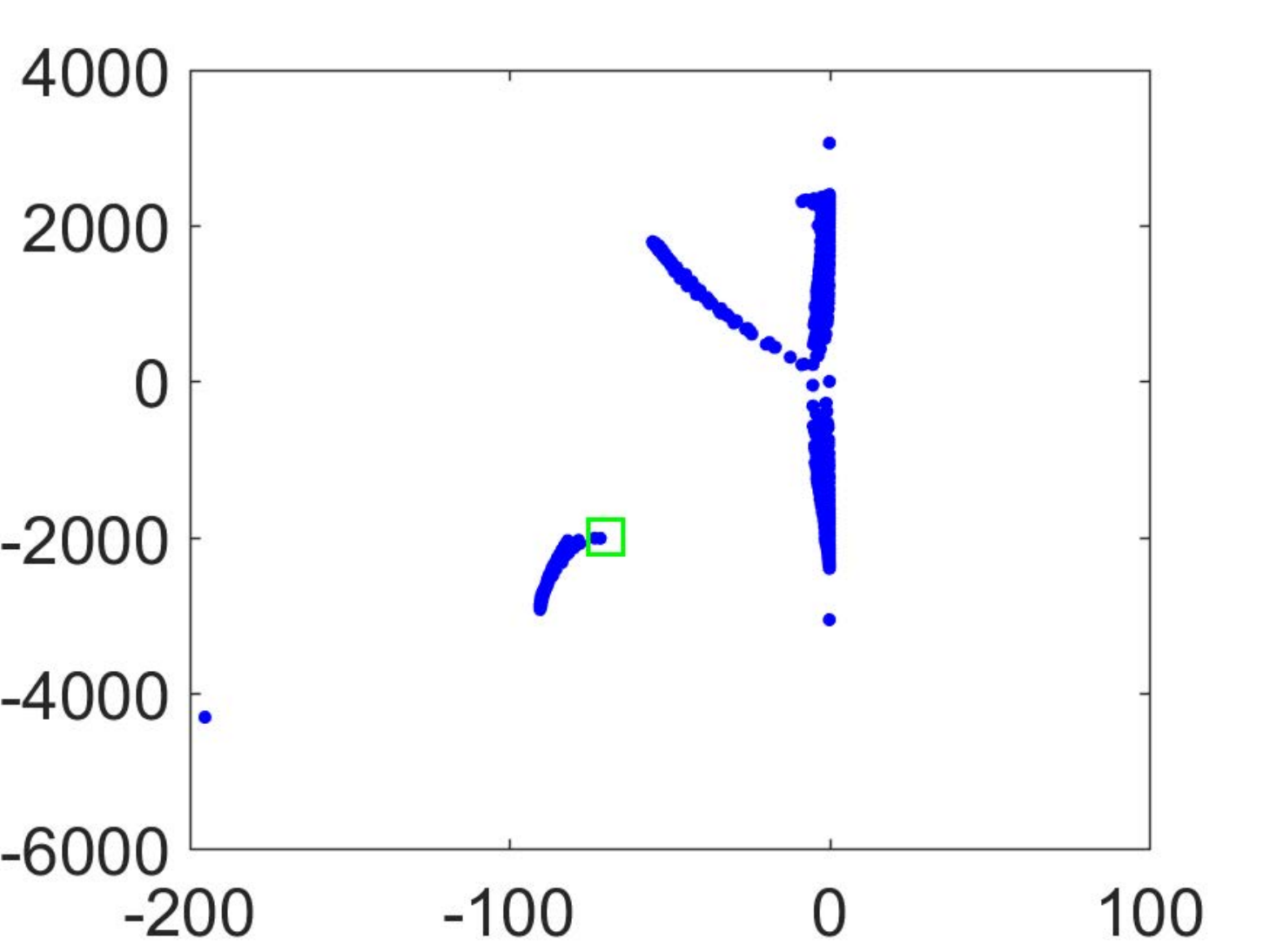} \
  \includegraphics[scale=0.43]{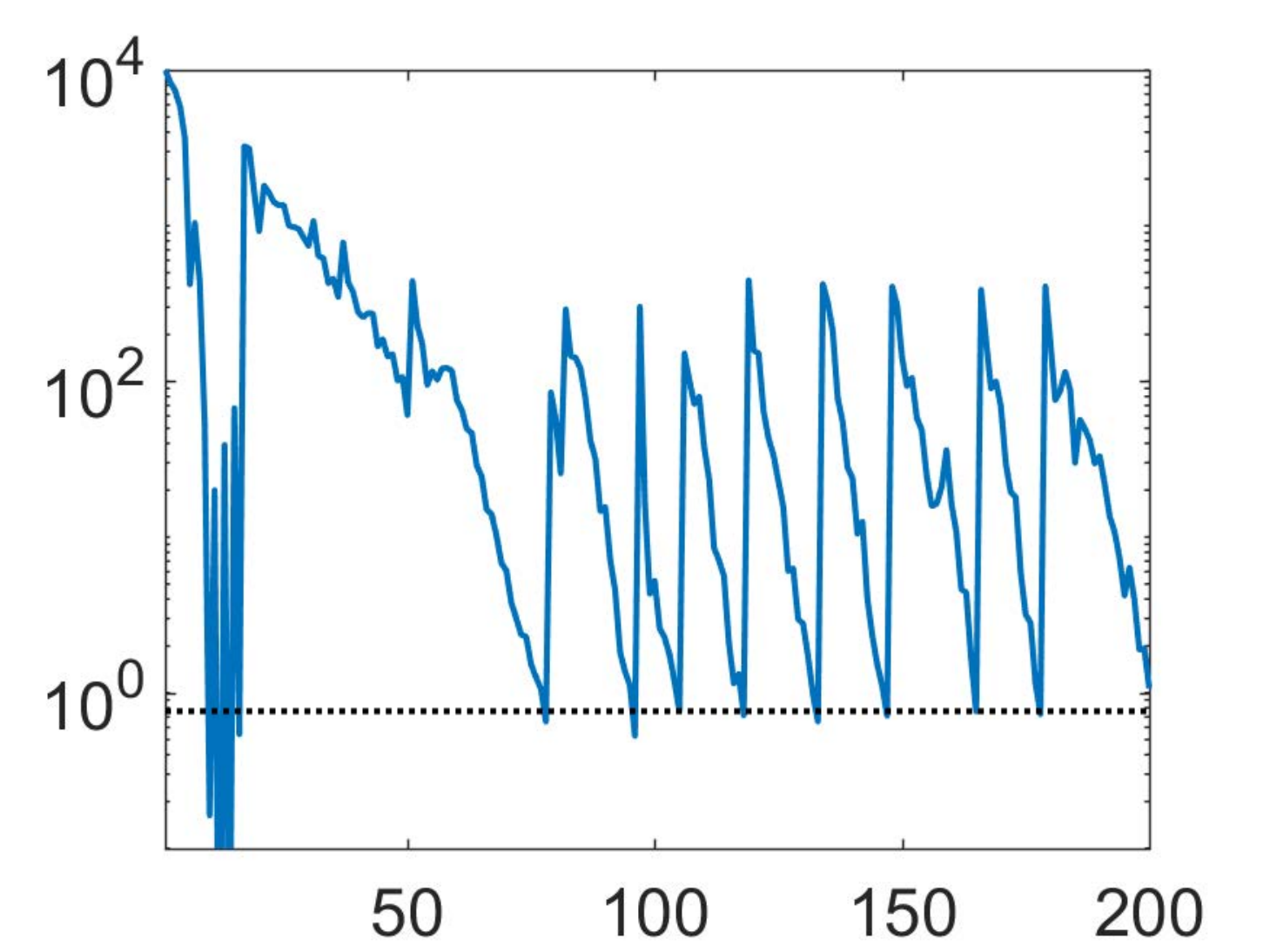}
\caption{(a) Spectrum and target of {\sf utrecht1331};
(b) Convergence history of 12 converged eigenpairs.}
\label{fig:1}
\end{figure}

Although the problem does not have infinite eigenvalues, we may also use the
homogeneous divided differences of Section~\ref{sec:homo}.
Note that this method is different from the standard divided difference.
In this case, we also find 12 eigenpairs in 200 iterations, after
10, 12, 15, 70, 81, 91, 107, 124, 143, 154, 170, and 190 iterations, respectively.
}
\end{experiment}

\begin{experiment}{\rm
We consider a popular challenge:
the problem {\sf gyroscopic}, a model of a gyroscopic dynamical system,
of size $n = 10000$; cf.~\cite[p.~654]{BSu05}.
Here, $A$ is diagonal with elements uniformly from $[0,1]$ with additionally $a_{11} = 0$,
$B$ is tridiagonal with $-1$s on the subdiagonal and $1$s on the superdiagonal,
and $C$ is diagonal with elements uniformly from $(-1,0)$.
Therefore, $A$ is symmetric positive semidefinite, $B$ is skew-symmetric,
and $C$ is symmetric negative definite, which is typical for this type of system.
The matrix $A$ is singular and the QEP has infinite eigenvalues.
Therefore, it seems appealing to exploit the homogeneous technique of Section~\ref{sec:homo}.
We take target $\tau = 80i$, and an exact LU preconditioner based on this target.
An eigenpair is considered converged if the residual is below $10^{-4}$.
All other parameters are as in Experiment~6.1.
We find 10 eigenpairs in 800 iterations, after
108, 109, 111, 115, 118, 143, 162, 176, 625, and 777 iterations, respectively.
Here, we see again the same pattern of first spending several iterations to obtain a good subspace,
then the quick detection of a number of eigenvalues, followed by
a new period of enriching the subspace before new eigenpairs are found.
}
\end{experiment}

\begin{experiment}
{\rm
For the next experiment, we take the largest cubic polynomial eigenvalue problem
$(\lambda^3 A_3 + \lambda^2 A_2 + \lambda A_1 + A_0)\,x = 0$
of the {\sf nlevp} toolbox \cite{nlevp}: the problem {\sf plasma\_drift}, with
coefficient matrices of size 512; see Figure~\ref{fig:2}.
We note that this spectrum is quite challenging, with close eigenvalue and eigenvalues
of high multiplicity.
Our target is $\tau = 0$, and as in the previous experiment we use an exact LU
preconditioner based on this target, so $LU = A_0$.
For the value extraction we use the two-dimensional minimum residual {\sf mr2} approach from \cite{HVo03}.
The other settings are the same as in Experiment~\ref{exp1}.
With $\eta = 0.1$, the Jacobi--Davidson method finds 19 eigenvalues in 200 outer iterations;
cf.~Fig.~\ref{fig:1}(c).
With respect to distance to the target, these are approximations to eigenvalues with index
1 through 12, 511, 512, 514, 14, 16, 515, and 510, respectively.
This ``alternating'' behavior is quite typical for iterative eigensolvers;
cf.~also \cite{HPl02,HKP05}.
The high indices can be explained by the fact that there are several eigenvalues
of high multiplicity close to the origin.
This illustrates that the selection method may work fine for problems with multiple
eigenvalues, as long as the computed eigenvalues are simple.
Other choices for $\eta$ result in 14 ($\eta = 0.01$),
10 ($\eta = 0.2$), and 11 ($\eta = 0.5$) eigenvalues.


\begin{figure}[!htbp]
\centering
  \includegraphics[scale=0.38]{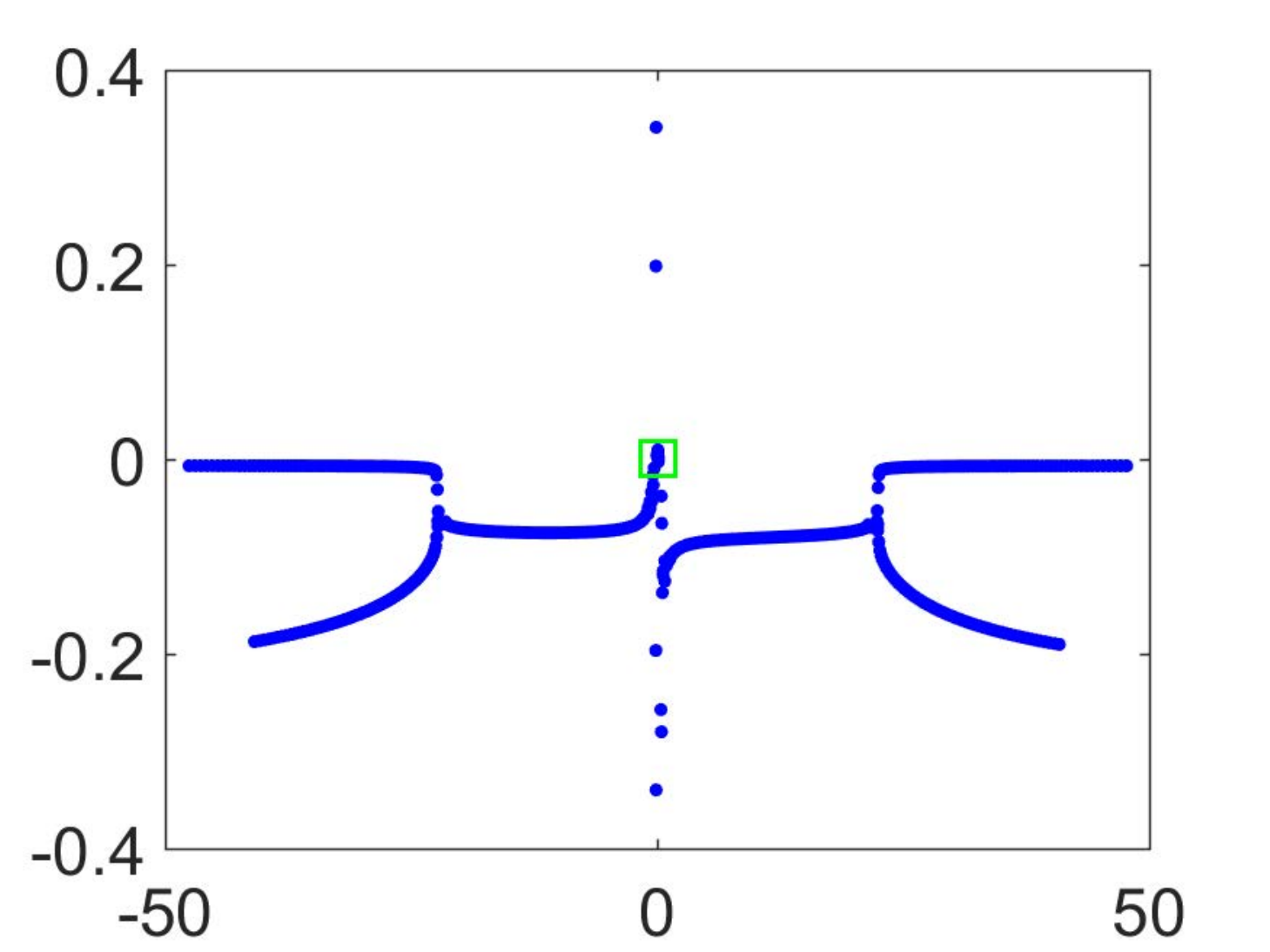} \
  \includegraphics[scale=0.38]{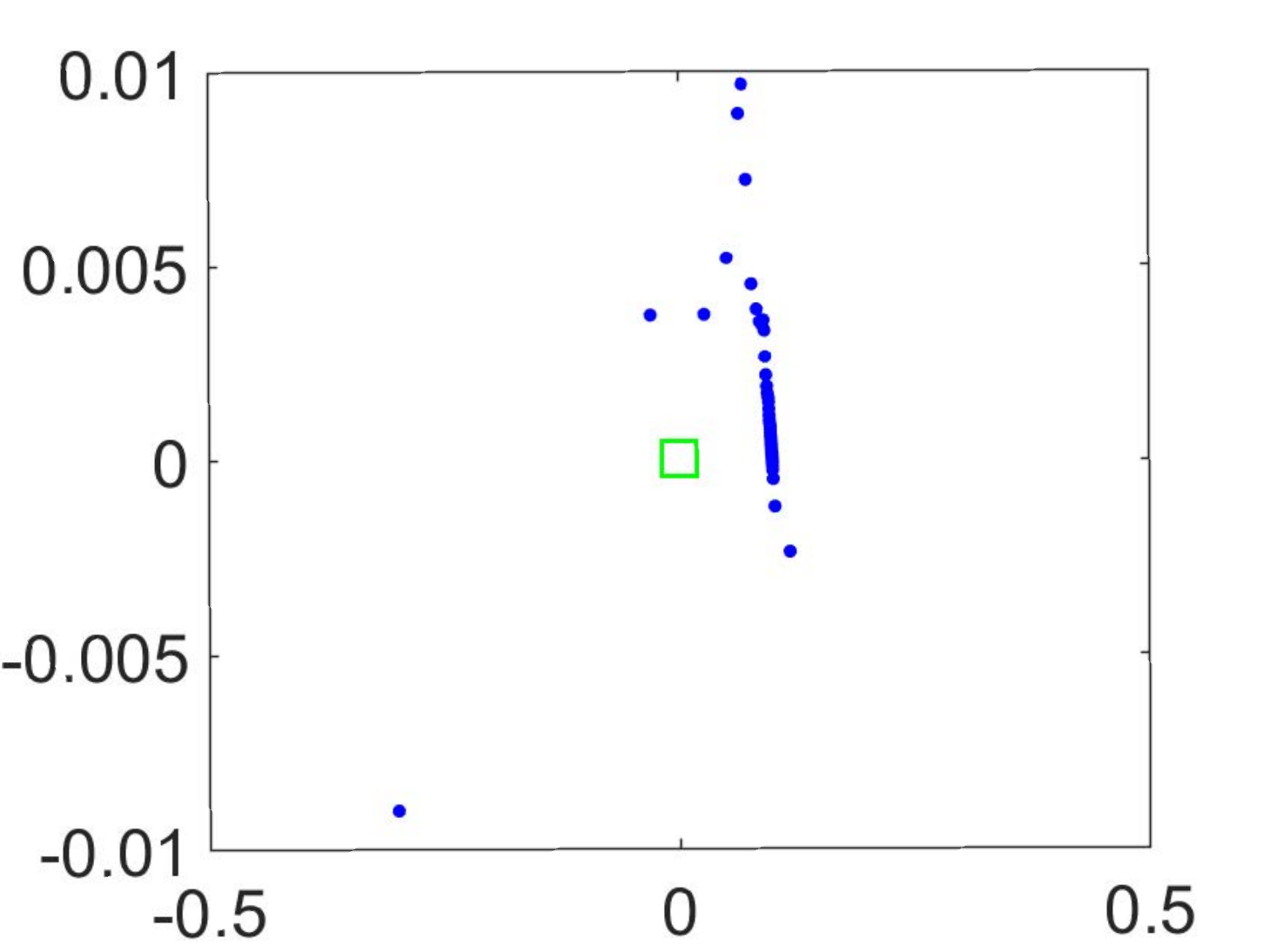}
  \includegraphics[scale=0.38]{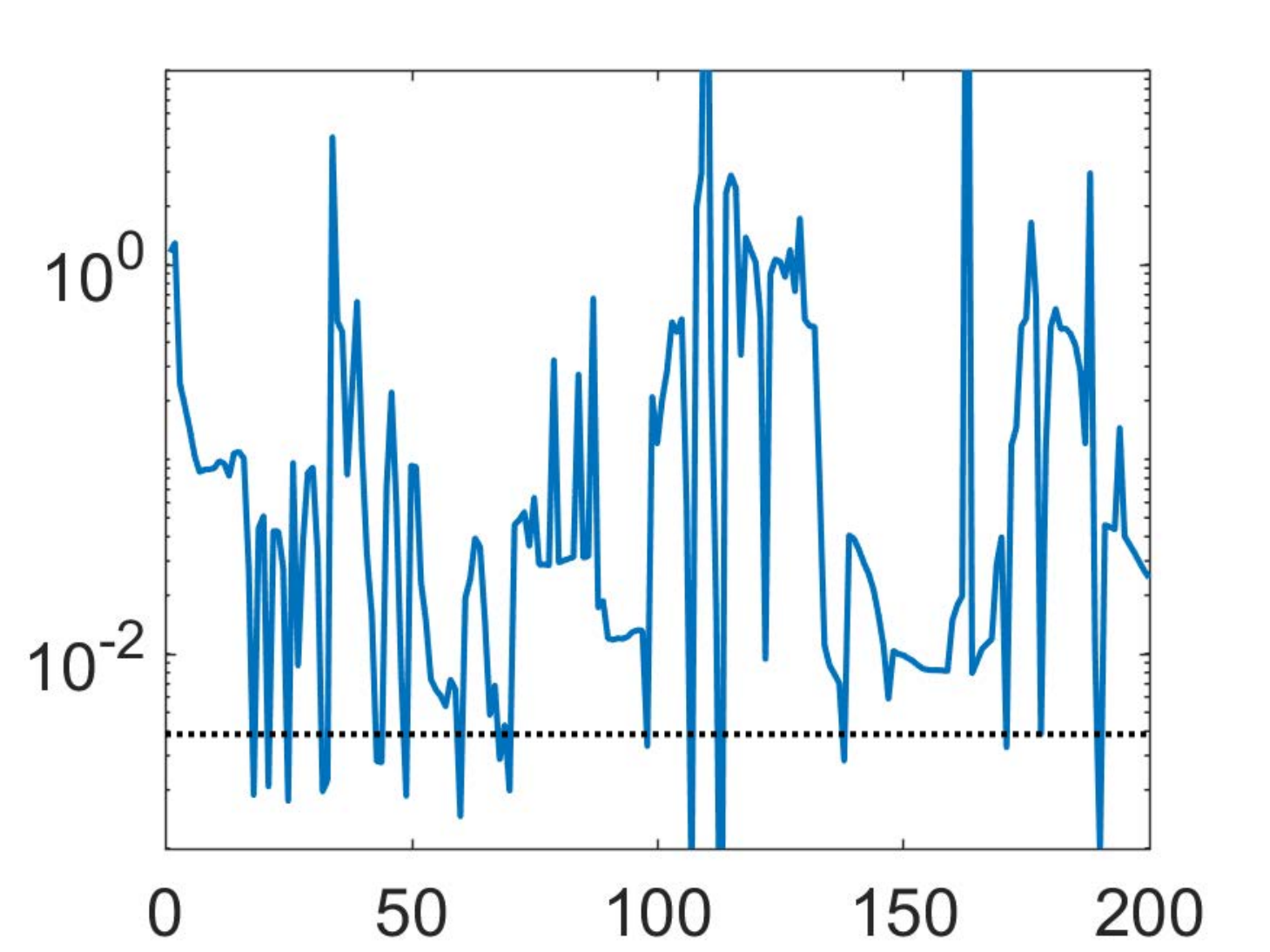}
\caption{(a) and (b): Spectrum and target of {\sf plasma\_drift};
(c) Convergence history of 19 converged eigenpairs.}
\label{fig:2}
\end{figure}

}
\end{experiment}

\begin{experiment}{\rm
We consider the 4-point boundary value problem
\begin{equation}\label{eq:stoss}
y''(x)+(\lambda + 2\mu \cos(x) +2\eta \cos(2x))\,y(x)=0,\quad y(0)=y(1)=y(2)=y(3)=0,
\end{equation}
where we seek
   $(\lambda,\mu,\eta)$ such that there exists a nonzero solution $y(x)$.
This problem can be decomposed into a 3-parameter eigenvalue problem that consists
of three 2-point boundary value problems of the form
\begin{equation}\label{eq:stoss3}
y_i''(x_i)+(\lambda + 2 \mu \cos(x_i) +2 \eta \cos(2x_i))\,y(x_i)=0,\quad
y_i(i-1)=y_i(i)=0
\end{equation}
for $i=1,2,3$.
A smooth function $y(x)$ that satisfies \eqref{eq:stoss}
can be constructed from the functions $y_1(x_1)$, $y_2(x_2)$, $y_3(x_3)$.
The 3-parameter eigenvalue problem \eqref{eq:stoss3} has the Klein oscillation property, which means that
for each triple of nonnegative integers $(m_1,m_2,m_3)$
there exist a triple of values $(\lambda,\mu,\eta)$ such that \eqref{eq:stoss} has a solution $y(x)$
that has $m_1$ zeros on interval $(0,1)$, $m_2$ zeros on $(1,2)$, and $m_3$ zeros on $(2,3)$.

We discretize \eqref{eq:stoss3} using the Chebyshev collocation on 200 points (cf.~\cite{HMMP19}),
which leads to an algebraic 3-parameter eigenvalue problem of the form
\begin{equation}
\label{eq:4point}
(A_i-\lambda B_i-\mu C_i-\eta D_i) \, x_i=0,\quad i=1,2,3.
\end{equation}
The solutions with indices $(j_1,j_2,j_3)$ such that $j_1+j_2+j_3$ is small correspond to eigenvalues
$(\lambda,\mu,\eta)$ close to $(0,0,0)$. To find eigenvalues close to the origin,
we apply the Jacobi--Davidson method, for
details see \cite{HMMP19}. We restrict the subspace
dimensions between 5 and 10 and solve the corresponding correction equations approximately by 10 steps of GMRES,
where we use the exact LU preconditioner based on the target, i.e., $A_j=L_jU_j$ for
$j = 1,2,3$.
The Jacobi--Davidson method returns 20 eigenvalues after performing 40 subspace updates. The first nine eigenvalues converged are provided
in Table \ref{tab:stoss} together with their indices, while the corresponding solutions $y(x)$ of \eqref{eq:stoss}
are illustrated in Figure \ref{fig:stoss}. Note that the indices in Table \ref{tab:stoss} confirm that the eigenvalues
converged are indeed the ones closest to the origin.

\begin{table}[htb]
\caption{The first 9 eigenvalues of the 4-point boundary value
problem \eqref{eq:stoss} retrieved by the Jacobi--Davidson method
with the origin as the target point.\label{tab:stoss}}
\begin{center}
{\footnotesize
\begin{tabular}{rrrrrr}
 \hline
\multicolumn{1}{c}{$\lambda$} & \multicolumn{1}{c}{$\mu$} & \multicolumn{1}{c}{$\eta$} & $j_1$ & $j_2$ & $j_3$ \\
\hline \rule{0pt}{2.3ex}%
 9.86960440 & $-0.00000000$ &  0.00000000   & 0 & 0 & 0\\
17.38523159 &  2.12527575   &$-12.73290564$ & 0 & 1 & 0\\
19.68377612 &  8.41730432   &  6.17620916   & 1 & 0 & 0\\
21.44695005 &$-10.07354787$ &  5.66869884   & 0 & 0 & 1\\
27.85962272 & 10.19955145   & $-6.02172707$ & 1 & 1 & 0\\
29.79885232 & $-8.32972041$ & $-6.38665167$ & 0 & 1 & 1\\
31.75591668 & $-1.66950908$ & 11.70626000   & 1 & 0 & 1\\
39.47841760 &  0.00000000   & $-0.00000000$ & 1 & 1 & 1\\
22.26126463 &  7.52057950   &$-38.93555514$ & 0 & 2 & 0\\
\hline
\end{tabular}}
\end{center}
\end{table}

\begin{figure}[htb]
\begin{center}
\includegraphics[scale=0.068]{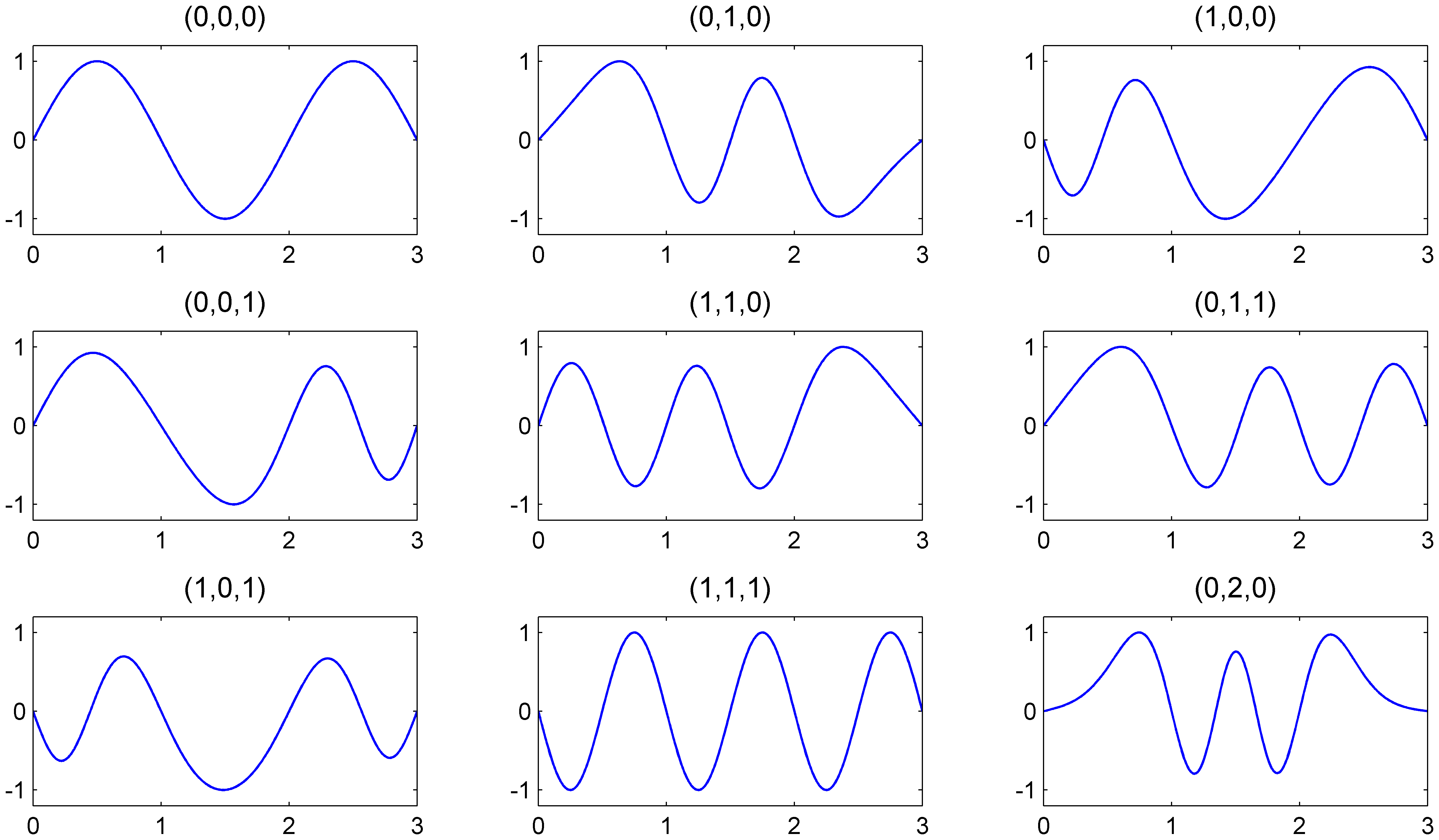}
\end{center}
\vspace{-0.5em}
\caption{First 9 solutions of \eqref{eq:stoss} corresponding to the eigenvalues
listed in Table \ref{tab:stoss}.}
\label{fig:stoss}
\vspace{-1.5em}
\end{figure}
}

\end{experiment}

\section{Conclusions}
\label{sec:concl}
We have presented several {\em selection criteria} for computing several eigenvalues for
nonlinear one-parameter, and linear and nonlinear multiparameter eigenvalue problems.
Selection means that an approximate eigenpair is picked
from candidate pairs that satisfy a certain suitable criterion.
The goal of this process is to steer the process away from already previously found pairs.
These criteria are easy to understand and implement, and also elegantly extend
to various types of eigenproblems.
We have also developed a divided difference and selection criterion
in homogeneous coordinates. This not only has the potential to handle infinite eigenvalues,
but also is a valuable alternative approach in itself.

The methods work directly on the original problem; no linearizations
(as for instance discussed in \cite{HMT13}) are necessary.
They require the computation of the left eigenvector, which implies some extra
costs for nonsymmetric problems. However, these additional costs are often
relatively small compared to the total costs. For certain problems with structure,
such as symmetric problems, the left eigenvectors come for free.
Also, left eigenvectors provide valuable information on the condition number
and reliability of the computed eigenvalues.

A main advantage of the selection techniques is that the search spaces
effectively may contain more useful vectors for the computations of
new eigenvectors.
Instead of locking, which keeps the converged vectors in the search space, the
search spaces can now be more fully used for new information.
Moreover, and also important for practical use, the selection criteria are relatively
easy to understand and implement compared with several existing approaches.

For the quadratic and polynomial (one-parameter) eigenvalue problem, the presented methods
are new, and a valuable alternative to other methods such as locking or block methods (cf.~\cite{Eff13});
a more detailed comparison can be found in Section~\ref{sec:compare}.
For linear and nonlinear {\em multiparameter} eigenvalue problems, we would like to stress the fact
that the presented selection techniques seem to be the {\em only} realistic option.
While for multiparameter eigenvalue problems we already proposed
selection criteria in the past, in this paper we propose updated
and less strict criteria of the type \eqref{mepnew} instead of \eqref{mepstrict}.

The approach can also be applied to general nonlinear eigenproblems $F(\lambda)\bx=\zero$,
as long as we can evaluate the derivative $F'(\lambda)$ and the divided difference
$F[\lambda,\mu]$.

We note that for challenging problems, it sometimes is not easy to find more
than about 10 eigenpairs with the selection criterion.
Reasons for this may be that a larger part of the search space is occupied by already
detected eigenvectors, or that the preconditioner is of lower quality for the new eigenvalues.
In this case, it may be a good idea to start a new process with a modified target and preconditioner.

Code for the proposed techniques for one-parameter eigenvalue problems is available on request;
for multiparameter eigenvalue problems, we refer to \cite{PleMEP}.



\bigskip\noindent{\bf Acknowledgments: }
We thank Daniel Kressner for helpful discussions and two expert referees
for useful comments.


\def\cprime{$'$}

\end{document}